\RequirePackage{ifthen}
\RequirePackage{ifpdf}
\newcommand{\driverOption}{}
\ifthenelse{\boolean{pdf}}{
  \renewcommand{\driverOption}{pdftex}
} { 
  \renewcommand{\driverOption}{dvips}
}

\documentclass[reqno, 11pt, letterpaper, oneside, commented, \driverOption]{amsart}

\newboolean{isCommented}
\usepackage{geometry}
\usepackage{bbm}
\usepackage[final]{graphicx}
\usepackage{upref}
\usepackage{enumitem}
\usepackage{latexsym}
\usepackage{amssymb}
\usepackage[ansinew]{inputenc}
\usepackage[T1]{fontenc}
\usepackage{lmodern,microtype}
\usepackage{mathtools}
\usepackage[ruled,vlined,norelsize]{algorithm2e}
\usepackage{mdframed}
\usepackage{xpatch}
\usepackage{fancyvrb}
\usepackage{mycommands}
\usepackage{todonotes}

\newcommand{\hyperrefDriverOption}{}
\ifthenelse{\boolean{pdf}}{
	\renewcommand{\hyperrefDriverOption}{pdftex}
} { 
	\renewcommand{\hyperrefDriverOption}{hypertex}
}
\usepackage[\hyperrefDriverOption,
  colorlinks = false,
  pdftitle={Gray codes and symmetric chains}]
  {hyperref}
\hypersetup{pdfauthor={Petr Gregor, Sven J\344ger, Torsten M\374tze, Joe Sawada, Kaja Wille}}

\usepackage{bookmark} 

\usepackage{tikz}
\usetikzlibrary{shapes}
\usetikzlibrary{arrows}
\usetikzlibrary{decorations.markings}
\usetikzlibrary{calc}
\usetikzlibrary{matrix}
\usetikzlibrary{backgrounds}
\tikzset{node_white/.style={circle, draw, fill=white, inner sep=0pt, text width=0pt, text height=0pt, text depth=0pt, minimum size = 4pt}}
\tikzset{node_black/.style={circle, draw, fill=black, inner sep=0pt, text width=0pt, text height=0pt, text depth=0pt, minimum size = 4pt}}
\tikzset{marker_red/.style={rectangle, draw, fill=red, inner sep=0pt, text width=0pt, text height=0pt, text depth=0pt, minimum size = 7pt}}
\tikzset{marker_blue/.style={diamond, draw, fill=blue, inner sep=0pt, text width=0pt, text height=0pt, text depth=0pt, minimum size = 9pt}}
\tikzset{node_huge/.style={circle, draw, fill=white, inner sep=0pt, text width=0pt, text height=0pt, text depth=0pt, minimum size = 33pt}}

\tikzset{line_solid/.style={draw, thick}}
\tikzset{line_dashed/.style={draw, thick, dashed}}
\tikzset{line_dotted/.style={draw, thick, dotted}}
\tikzset{line_gray_bg/.style={draw=black!40!white, line width=7pt, line cap=round, line join=round}}
\tikzset{dashdot/.style={dash pattern=on .8pt off 2pt on 4pt off 2pt}}
\tikzset{dotdot/.style={dash pattern=on 1pt off 1pt}}
\tikzset{ptr/.style={decoration={markings,mark=at position 1 with {\arrow[scale=1.5,>=latex]{>}}},postaction={decorate}}}
    
\ifthenelse{\boolean{isCommented}} {
  \newcommand{\PG}[1]{\marginpar{\parbox{4cm}{{\small {\bf PG:} #1}}}}
  \newcommand{\TM}[1]{\marginpar{\parbox{4cm}{{\small {\bf TM:} #1}}}}
  \newcommand{\SJ}[1]{\marginpar{\parbox{1.5cm}{{\tiny {\bf SJ:} #1\par}}}}
} { 
  \newcommand{\PG}[1]{}
  \newcommand{\TM}[1]{}
  \newcommand{\SJ}[1]{}
}

\newtheorem{theorem}{Theorem}
\newtheorem{lemma}[theorem]{Lemma}

\newtheorem{proposition}[theorem]{Proposition}

\newtheorem{problem}{Problem}

\theoremstyle{definition}

\theoremstyle{remark}

\ifthenelse{\boolean{isCommented}} {
	\geometry{
	  hmargin={25mm, 25mm},
	  marginparwidth=15mm,
	  vmargin={25mm, 25mm},
	  headsep=10mm,
	  headheight=5mm,
	  footskip=10mm
	}
} { 
	\geometry{
	  hmargin={25mm, 25mm},
	  vmargin={25mm, 25mm},
	  headsep=10mm,
	  headheight=5mm,
	  footskip=10mm
	}
}

\makeatletter
\g@addto@macro{\endabstract}{\@setabstract}
\newcommand{\authorfootnotes}{\renewcommand\thefootnote{\@fnsymbol\c@footnote}}%
\makeatother

\begin{document}

\begin{center}

\authorfootnotes
\LARGE Gray codes and symmetric chains\footnote{An extended abstract of this paper appeared in~\cite{DBLP:conf/icalp/GregorJMSW18}.}
\vspace{2mm}

\large
Petr~Gregor\footnote{E-Mail: \texttt{gregor@ktiml.mff.cuni.cz}. Petr Gregor was supported by GACR grant GA~19-08554S.}\textsuperscript{1},
Sven~J\"ager\footnote{E-Mail: \texttt{jaeger@math.tu-berlin.de}}\textsuperscript{2},
Torsten~M\"utze\footnote{E-Mail: \texttt{torsten.mutze@warwick.ac.uk}. Torsten M\"utze is also affiliated with Charles University, Faculty of Mathematics and Physics, and was supported by GACR grant GA~19-08554S, and by DFG grant~413902284.}\textsuperscript{3},
Joe Sawada\footnote{E-Mail: \texttt{jsawada@uoguelph.ca}}\textsuperscript{4},
Kaja Wille\footnote{E-Mail: \texttt{wille@math.tu-berlin.de}}\textsuperscript{2}
\setcounter{footnote}{0}
\bigskip

\small

\textsuperscript{1}Department of Theoretical Computer Science and Mathematical Logic, \\ Charles University, Prague, Czech Republic \par
\textsuperscript{2}Institut f\"{u}r Mathematik, Technische Universit\"{a}t Berlin, Germany \par
\textsuperscript{3}Department of Computer Science, University of Warwick, United Kingdom \par
\textsuperscript{4}School of Computer Science, University of Guelph, Canada \par

\bigskip

\begin{minipage}{0.8\linewidth}
\textsc{Abstract.}
We consider the problem of constructing a cyclic listing of all bitstrings of length~$2n+1$ with Hamming weights in the interval $[n+1-\ell,n+\ell]$, where $1\leq \ell\leq n+1$, by flipping a single bit in each step.
This is a far-ranging generalization of the well-known middle two levels problem (the case~$\ell=1$).
We provide a solution for the case~$\ell=2$, and we solve a relaxed version of the problem for general values of~$\ell$, by constructing cycle factors for those instances.
The proof of the first result uses the lexical matchings introduced by Kierstead and Trotter, which we generalize to arbitrary consecutive levels of the hypercube.
The proof of the second result uses symmetric chain decompositions of the hypercube, a concept known from the theory of posets.
We also present several new constructions of such decompositions based on lexical matchings.
In particular, we construct four pairwise edge-disjoint symmetric chain decompositions of the $n$-dimensional hypercube for any~$n\geq 12$.
\end{minipage}

\vspace{2mm}

\begin{minipage}{0.8\linewidth}
\end{minipage}

\vspace{2mm}

\end{center}

\vspace{2mm}

\section{Introduction}

Gray codes are named after Frank Gray, a researcher at Bell Labs, who described a simple method to generate all $2^n$~bitstrings of length~$n$ by flipping a single bit in each step~\cite{gray:patent}, now known as the binary reflected Gray code.
This code found widespread use, e.g., in circuit design and testing, signal processing and error correction, data compression etc.; many more applications are mentioned in the survey~\cite{MR1491049}.
The binary reflected Gray code is also implicit in the well-known \emph{Towers of Hanoi} puzzle and the \emph{Chinese ring} puzzle that date back to the 19th century.
The theory of Gray codes has developed considerably in the last decades, and the term is now used more generally to describe an exhaustive listing of any class of combinatorial objects where successive objects in the list differ by a small amount.
In particular, such generation algorithms have been developed for several fundamental combinatorial objects of interest for computer scientists, such as bitstrings, permutations, partitions, trees etc., all of which are covered in depth in the most recent volume of Knuth's seminal series \emph{The Art of Computer Programming}~\cite{MR3444818}.

Since the discovery of the binary reflected Gray code, there has been continued interest in developing Gray codes for bitstrings of length~$n$ that satisfy various additional constraints.
For instance, a Gray code with the property that each bit is flipped (almost) the same number of times was first constructed by Bakos~\cite[p.~28--37]{MR0242572}.
Goddyn and Gvozdjak~\cite{MR2014514} constructed an $n$-bit Gray code in which any two successive flips of the same bit are almost $n$~steps apart, which is best possible.
These are only two examples of a large body of work on possible Gray code transition sequences; see also~\cite{MR1377601, MR2427745, MR2974271}.
Savage and Winkler~\cite{MR1329390} constructed a Gray code that generates all $2^n$~bitstrings such that all bitstrings with Hamming weight~$k$ appear before all bitstrings with weight~$k+2$, for each $0\leq k\leq n-2$, where the \emph{Hamming weight} of a bitstring is the number of its 1-bits.
They used this construction to tackle the infamous \emph{middle two levels problem}, which asks for a cyclic listing of all bitstrings of length~$2n+1$ with weights in the interval~$[n,n+1]$ by flipping a single bit in each step.
This problem was raised in the 1980s and received considerable attention in the literature  (a detailed historic account is given in~\cite{MR3483129}).
A general existence proof for such a Gray code for any~$n\geq 1$ has been found only recently~\cite{MR3483129, gregor-muetze-nummenpalo:18}, and an algorithm for computing it using $\cO(1)$ amortized time and $\cO(n)$ space was subsequently presented in~\cite{MR4075363}.
The starting point of this work is the following more general problem raised independently by Buck and Wiedemann~\cite{MR737262},
Savage~\cite{MR1275228}, Gregor and {\v{S}}krekovski~\cite{MR2609124}, and by Shen and Williams~\cite{shen_williams_2015}.

\renewcommand\theproblem{\Alph{problem}}
\setcounter{problem}{12}  

\begin{problem}[middle $2\ell$~levels problem]
\label{prob:gmlc}
For any~$n\geq 1$ and $1\leq \ell\leq n+1$, construct a cyclic listing of all bitstrings of length~$2n+1$ with Hamming weights in the interval $[n+1-\ell,n+\ell]$ by flipping a single bit in each step.
\end{problem}

The special case~$\ell=1$ of Problem~\ref{prob:gmlc} is the middle two levels problem mentioned before.
The case~$\ell=n+1$ is solved by the binary reflected Gray code discussed in the beginning.
Moreover, the cases~$\ell=n$ and~$\ell=n-1$ were settled in~\cite{MR737262, MR1887372, locke-stong:03} and~\cite{MR2609124}, respectively.

A natural framework for studying such Gray code problems is the \emph{$n$-dimensional hypercube~$Q_n$}, or $n$-cube for short, the graph formed by all bitstrings of length~$n$, with an edge between any two bitstrings that differ in exactly one bit.
The 5-cube is illustrated in Figure~\ref{fig:q5}~(a).
The \emph{$k$th level} of the $n$-cube is the set of all bitstrings with Hamming weight exactly~$k$.
In this terminology, Problem~\ref{prob:gmlc} asks for a Hamilton cycle in the subgraph of the $(2n+1)$-cube induced by the middle $2\ell$~levels.

\begin{figure}
\makebox[0cm]{ 
\includegraphics[scale=0.916]{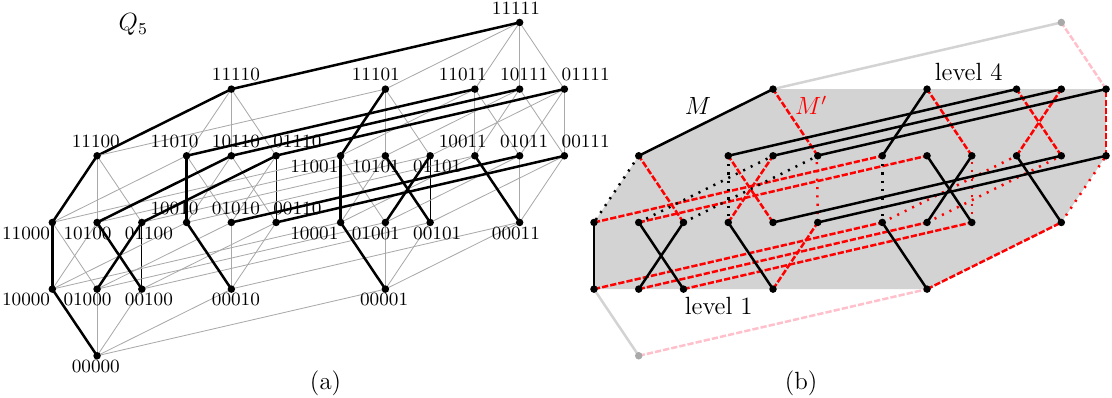}
}
\caption{(a) The 5-cube with the (standard) symmetric chain decomposition~$\cD_0$, where the edges along the chains are highlighted by thick lines.
(b) Building a cycle factor through the middle four levels of the 5-cube as explained in the proof of Theorem~\ref{thm:gmlc-2fac} with SCDs~$\cD:=\cD_0$ (black) and~$\cD':=\ol{\cD_0}$ (red).
The edges that are removed from~$\cD$ and~$\cD'$ are dotted, so the solid and dashed edges are the two matchings~$M$ and~$M'$ whose union forms the cycle factor.
It has three cycles of lengths 4, 4 and 22, visiting all 30 bitstrings with Hamming weight in the interval~$[1,4]$.
}
\label{fig:q5}
\end{figure}

The most general version of this problem is whether the subgraph of the $n$-cube induced by all levels in an arbitrary weight interval~$[a,b]$ has a Hamilton cycle.
However, unless we are in the cases covered by Problem~M (odd dimension and symmetric levels around the middle) or $n$ is even and $(a,b)=(0,n)$, the corresponding subgraph of the $n$-cube has two partition classes of different sizes, and thus cannot have a Hamilton cycle.
Nonetheless, we may still ask for a cycle that visits all vertices in the smaller partition class, or for a cyclic listing of all vertices in which only few transitions flip two instead of one bit, where `few' means the difference in size between the two partition classes.
Both of these are natural generalizations of a Hamilton cycle, and we will refer to them as an `almost' Hamilton cycle.
This generalized problem was solved in~\cite{MR3758308} for several values of~$n\geq 1$ and $0\leq a\leq b\leq n$, and it was shown that a solution to Problem~\ref{prob:gmlc} would imply this result for all possible values of~$n\geq 1$ and $0\leq a\leq b\leq n$.

\subsection{Our results}

In this work we solve the case~$\ell=2$ of Problem~\ref{prob:gmlc}, i.e., we construct a cyclic listing of all bitstrings of length~$2n+1$ with Hamming weights in the interval~$[n-1,n+2]$.

\begin{theorem}
\label{thm:gmlc4}
For any~$n\geq 1$, the subgraph of the $(2n+1)$-cube induced by the middle four levels has a Hamilton cycle.
\end{theorem}

Combining Theorem~\ref{thm:gmlc4} with the results from~\cite{MR3758308} shows more generally that the subgraph of the $n$-cube induced by any four consecutive levels has an `almost' Hamilton cycle.

As another partial result towards Problem~\ref{prob:gmlc}, we show that the subgraph of the $(2n+1)$-cube induced by the middle $2\ell$~levels has a cycle factor.
A \emph{cycle factor} is a collection of disjoint cycles which together visit all vertices of the graph.
In particular, a Hamilton cycle is a cycle factor consisting only of a single cycle.
Note here that the existence of a cycle factor for general values of~$\ell$ is not an immediate consequence of Hall's theorem, which is applicable only for~$\ell=1$ and~$\ell=n+1$, as only in those cases all vertices of the underlying graph have the same degree.

\begin{theorem}
\label{thm:gmlc-2fac}
For any~$n\geq 1$ and $1\leq \ell\leq n+1$, the subgraph of the $(2n+1)$-cube induced by the middle $2\ell$~levels has a cycle factor.
\end{theorem}

Our proof of Theorem~\ref{thm:gmlc-2fac} is concise and illustrative, and it motivates the subsequent discussion, so we present it right now.
It uses a well-known concept from the theory of partially ordered sets (posets), a so-called symmetric chain decomposition.
Here we define this term for the $n$-cube using graph-theoretic language.
A \emph{symmetric chain} in~$Q_n$ is a path $(x_k,x_{k+1},\ldots,x_{n-k})$ in the $n$-cube where $x_i$ is from level~$i$ for all $k\leq i\leq n-k$, and a \emph{symmetric chain decomposition}, or SCD for short, is a partition of the vertices of~$Q_n$ into symmetric chains.
For illustration, an SCD of~$Q_5$ is shown in Figure~\ref{fig:q5}~(a).
We say that two SCDs are \emph{edge-disjoint} if the corresponding sets of paths are edge-disjoint, i.e., if there are no two consecutive vertices in a chain of the first SCD that are also contained in a chain of the second SCD.
There is a well-known construction of two edge-disjoint SCDs in the $n$-cube for any~$n\geq 1$~\cite{MR532807}, which we will discuss momentarily.

\begin{proof}[Proof of Theorem~\ref{thm:gmlc-2fac}]
The proof is illustrated in Figure~\ref{fig:q5}~(b).
Consider two edge-disjoint SCDs~$\cD$ and~$\cD'$ in the $(2n+1)$-cube.
Let~$\cR$ and~$\cR'$ be the chains obtained from~$\cD$ and~$\cD'$, respectively, by restricting them to the middle $2\ell$~levels, so chains that are longer than~$2\ell-1$ get shortened on both sides.
As all chains in~$\cR$ and~$\cR'$ start and end at symmetric levels and the dimension~$2n+1$ is odd, all these paths have odd length (possible lengths are $1,3,\ldots,2\ell-1$).
Therefore, by taking every second edge on every path from~$\cR$ and~$\cR'$, we obtain two perfect matchings~$M$ and~$M'$ in the subgraph of the $(2n+1)$-cube induced by the middle $2\ell$~levels.
As the paths in~$\cR$ and~$\cR'$ are edge-disjoint, the matchings~$M$ and~$M'$ are also edge-disjoint.
Therefore, the union of~$M$ and~$M'$ is the desired cycle factor.
\end{proof}

This proof motivates the search for a large collection of pairwise edge-disjoint SCDs in the $n$-cube.
We can then use any two of them to construct a cycle factor as described in the previous proof, and use this cycle factor as a starting point for building a Hamilton cycle.
This two-step approach of building a Hamilton cycle via a cycle factor proved to be very successful for such problems (see e.g.~\cite{MR2548540, MR2836824, MR2925746, MR3483129, MR3599935, DBLP:conf/soda/SawadaW18, muetze-nummenpalo-walczak:21}).
Consequently, for the rest of this section we focus on edge-disjoint SCDs in the $n$-cube.

There is a well-known construction of an SCD for the $n$-cube that is best described by the following parenthesis matching approach pioneered by Greene and Kleitman~\cite{MR0389608}; see Figure~\ref{fig:paren}.
For any vertex~$x$ of the $n$-cube, we interpret the 0s in~$x$ as opening brackets and the 1s as closing brackets.
By matching closest pairs of opening and closing brackets in the natural way, the chain containing~$x$ is obtained by flipping the leftmost unmatched 0 to move up the chain, or the rightmost unmatched~1 to move down the chain, until no more unmatched bits can be flipped.
It is easy to see that this indeed yields an SCD of the $n$-cube for any~$n\geq 1$. 
We denote this standard SCD by~$\cD_0$; it is shown in Figure~\ref{fig:q5}~(a) for~$n=5$.
There are several alternative ways to describe this SCD (see~\cite{MR0043115, MR0319772, MR0450151}).

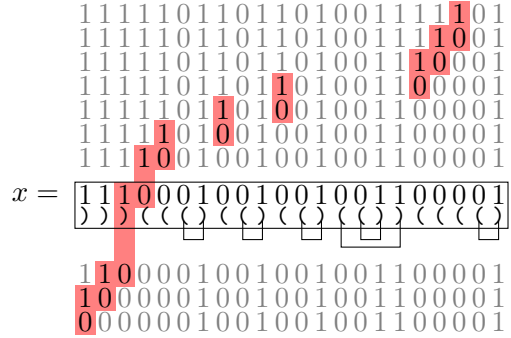
\begin{figure}
\begin{tikzpicture}[every node/.style={gray, inner sep=1pt}]
 \pgfsetmatrixcolumnsep{0pt}
 \pgfsetmatrixrowsep{0pt}
 \matrix (m) [matrix of nodes] {
  1&1&1&1&1&0&1&1&0&1&1&0&1&0&0&1&1&1&1&|[black]|1&0&1 \\
  1&1&1&1&1&0&1&1&0&1&1&0&1&0&0&1&1&1&|[black]|1&|[black]|0&0&1 \\
  1&1&1&1&1&0&1&1&0&1&1&0&1&0&0&1&1&|[black]|1&|[black]|0&0&0&1 \\
  1&1&1&1&1&0&1&1&0&1&|[black]|1&0&1&0&0&1&1&|[black]|0&0&0&0&1 \\
  1&1&1&1&1&0&1&|[black]|1&0&1&|[black]|0&0&1&0&0&1&1&0&0&0&0&1 \\
  1&1&1&1&|[black]|1&0&1&|[black]|0&0&1&0&0&1&0&0&1&1&0&0&0&0&1 \\
  1&1&1&|[black]|1&|[black]|0&0&1&0&0&1&0&0&1&0&0&1&1&0&0&0&0&1 \\[5pt]
  |[black]|1&|[black]|1&|[black]|1&|[black]|0&|[black]|0&|[black]|0&|[black]|1&|[black]|0&|[black]|0&|[black]|1&|[black]|0&|[black]|0&|[black]|1&|[black]|0&|[black]|0&|[black]|1&|[black]|1&|[black]|0&|[black]|0&|[black]|0&|[black]|0&|[black]|1 \\[-1pt]
  |[black]|\footnotesize\texttt)&|[black]|\footnotesize\texttt)&|[black]|\footnotesize\texttt)&|[black]|\footnotesize\texttt(&|[black]|\footnotesize\texttt(&|[black]|\footnotesize\texttt(&|[black]|\footnotesize\texttt)&|[black]|\footnotesize\texttt(&|[black]|\footnotesize\texttt(&|[black]|\footnotesize\texttt)&|[black]|\footnotesize\texttt(&|[black]|\footnotesize\texttt(&|[black]|\footnotesize\texttt)&|[black]|\footnotesize\texttt(&|[black]|\footnotesize\texttt(&|[black]|\footnotesize\texttt)&|[black]|\footnotesize\texttt)&|[black]|\footnotesize\texttt(&|[black]|\footnotesize\texttt(&|[black]|\footnotesize\texttt(&|[black]|\footnotesize\texttt(&|[black]|\footnotesize\texttt) \\[13pt]
  1&|[black]|1&|[black]|0&0&0&0&1&0&0&1&0&0&1&0&0&1&1&0&0&0&0&1 \\
  |[black]|1&|[black]|0&0&0&0&0&1&0&0&1&0&0&1&0&0&1&1&0&0&0&0&1 \\
  |[black]|0&0&0&0&0&0&1&0&0&1&0&0&1&0&0&1&1&0&0&0&0&1 \\
 };
 \draw ($(m-9-6.south)+(0,0.1)$) -- ($(m-9-6.south)-(0,0.15)$) -- ($(m-9-7.south)-(0,0.15)$) -- ($(m-9-7.south)+(0,0.1)$);
 \draw ($(m-9-9.south)+(0,0.1)$) -- ($(m-9-9.south)-(0,0.15)$) -- ($(m-9-10.south)-(0,0.15)$) -- ($(m-9-10.south)+(0,0.1)$);
 \draw ($(m-9-12.south)+(0,0.1)$) -- ($(m-9-12.south)-(0,0.15)$) -- ($(m-9-13.south)-(0,0.15)$) -- ($(m-9-13.south)+(0,0.1)$);
 \draw ($(m-9-15.south)+(0,0.1)$) -- ($(m-9-15.south)-(0,0.15)$) -- ($(m-9-16.south)-(0,0.15)$) -- ($(m-9-16.south)+(0,0.1)$);
 \draw ($(m-9-14.south)+(0,0.1)$) -- ($(m-9-14.south)-(0,0.25)$) -- ($(m-9-17.south)-(0,0.25)$) -- ($(m-9-17.south)+(0,0.1)$); 
 \draw ($(m-9-21.south)+(0,0.1)$) -- ($(m-9-21.south)-(0,0.15)$) -- ($(m-9-22.south)-(0,0.15)$) -- ($(m-9-22.south)+(0,0.1)$);
 \draw (m-8-1.north west) rectangle (m-9-22.south east);
 
\begin{scope}[on background layer]
 \fill[red!50] (m-8-3.north west) rectangle (m-10-3.south east);
 \fill[red!50] (m-10-2.north west) rectangle (m-11-2.south east);
 \fill[red!50] (m-11-1.north west) rectangle (m-12-1.south east);
 \fill[red!50] (m-7-4.north west) rectangle (m-8-4.south east);
 \fill[red!50] (m-6-5.north west) rectangle (m-7-5.south east);
 \fill[red!50] (m-5-8.north west) rectangle (m-6-8.south east);
 \fill[red!50] (m-4-11.north west) rectangle (m-5-11.south east);
 \fill[red!50] (m-3-18.north west) rectangle (m-4-18.south east);
 \fill[red!50] (m-2-19.north west) rectangle (m-3-19.south east);
 \fill[red!50] (m-1-20.north west) rectangle (m-2-20.south east);
\end{scope}

\node[left,anchor=east,black,xshift=-2mm] at (m-8-1) {$x={}$};
\end{tikzpicture}
\caption{The parenthesis matching approach for constructing the symmetric chain containing a bitstring~$x$, yielding the symmetric chain decomposition~$\cD_0$.
The highlighted bits are the leftmost unmatched~0 and the rightmost unmatched~1 in each bitstring.
}
\label{fig:paren}
\end{figure}

By taking complements, we obtain another SCD, which we denote by~$\ol{\cD_0}$.
It is not hard to see that~$\cD_0$ and~$\ol{\cD_0}$ are in fact edge-disjoint for any~$n\geq 1$~\cite{MR532807}.
Figure~\ref{fig:q5}~(b) shows both SCDs for~$n=5$, and how they are used for building a cycle factor.

Our next result is a simple construction of another SCD in the $n$-cube for \emph{even} values of~$n\geq 2$, which we call~$\cD_1$.
It has the additional feature that~$\cD_0$, $\ol{\cD_0}$, $\cD_1$ and~$\ol{\cD_1}$ are pairwise edge-disjoint for~$n\geq 6$.

\begin{theorem}
\label{thm:4scds-even}
For any even~$n\geq 6$, the $n$-cube contains four pairwise edge-disjoint SCDs.
\end{theorem}

Figure~\ref{fig:d01} shows the SCDs~$\cD_0$ and~$\cD_1$ in~$Q_6$.
Their complements~$\ol{\cD_0}$ and~$\ol{\cD_1}$ are not shown for clarity.
Note that four edge-disjoint SCDs are best possible for~$Q_6$, as they use up all edges incident with the middle level.

For odd values of~$n$, we can still construct four edge-disjoint SCDs in the $n$-cube (except in a few small cases).
However, the construction is not as direct and explicit as for even~$n$.

\begin{theorem}
\label{thm:4scds-odd}
For~$n=7$ and any odd~$n\geq 13$, the $n$-cube contains four pairwise edge-disjoint SCDs.
\end{theorem}

For odd~$n$, we can combine any two of the four edge-disjoint SCDs in the $n$-cube guaranteed by Theorem~\ref{thm:4scds-odd} to a cycle factor in the middle $2\ell$~levels, as explained before, yielding in total $\binom{4}{2}=6$ distinct cycle factors.
Our technique for proving Theorem~\ref{thm:4scds-odd} is the reason why the cases~$n=9$ and~$n=11$ are excluded in the statement of the theorem.
Specifically, we construct four edge-disjoint SCDs in the 7-cube in an ad hoc fashion and then apply the following product construction.

\begin{theorem}
\label{thm:prod}
If~$Q_a$ and~$Q_b$ each contain $k$~pairwise edge-disjoint SCDs, then~$Q_{a+b}$ contains~$k$ pairwise edge-disjoint SCDs.
\end{theorem}

Theorem~\ref{thm:prod} shows in particular that from~$k$ edge-disjoint SCDs in a hypercube of fixed dimension~$n$, we obtain $k$~edge-disjoint SCDs for infinitely many larger dimensions $2n,3n,4n,\ldots$.

We conjecture that the $n$-cube has $\lfloor n/2\rfloor+1$ pairwise edge-disjoint SCDs, but so far we only know that this holds for~$n\leq 7$.
Clearly, finding this many edge-disjoint SCDs would be best possible, as they use up all middle edges of the cube.
Maximum sets of pairwise edge-disjoint SCDs in the $n$-cube we found for $n=1,2,\ldots,11$ are shown in Table~\ref{tab:small-scds}, together with the aforementioned upper bound.

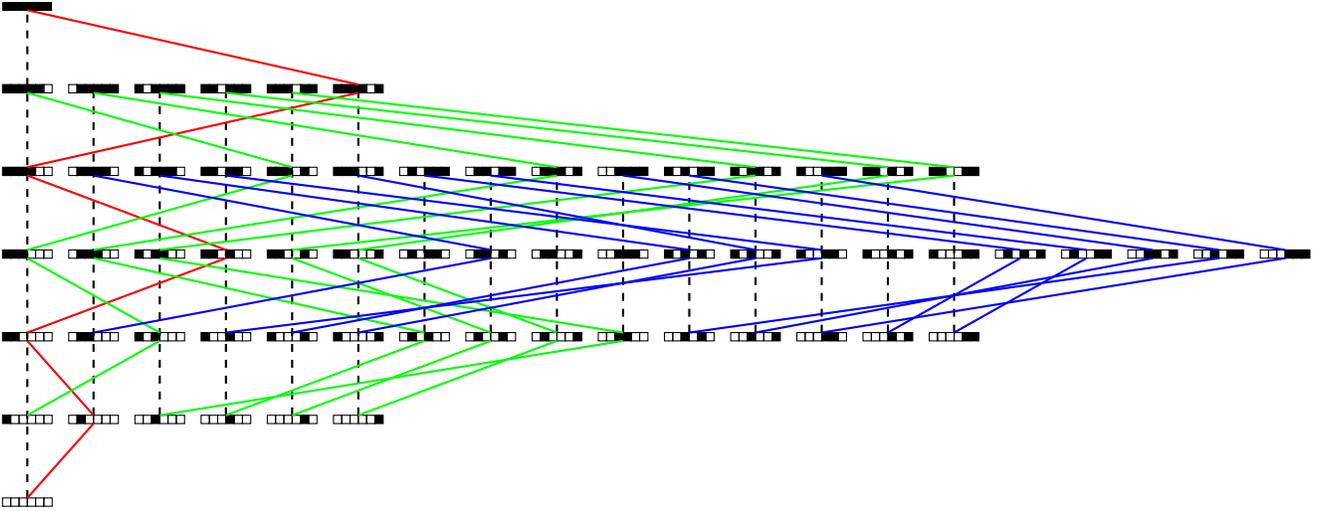
\begin{figure}
\makebox[0cm]{ 
\begin{tikzpicture}[scale=1.1]
\foreach \i in {0,...,5} {
\draw ({\i/10},0) rectangle ({(\i+1)/10},0.1);
}
\foreach \i in {0,...,5} {
\foreach \j in {0,...,5} {
\draw ({\i*0.8+\j/10},1) rectangle ({\i*0.8+(\j+1)/10},1.1);
}
\fill ({\i*9/10},1) rectangle ({(\i*9+1)/10},1.1);
}
\foreach \i in {0,...,14} {
\foreach \j in {0,...,5} {
\draw ({(\i*8+\j)/10},2) rectangle ({(\i*8+\j+1)/10},2.1);
}
}
\fill (0,2) rectangle (0.2,2.1);
\fill (0.9,2) rectangle (1.1,2.1);
\foreach \i in {2,...,5} {
\fill (\i*0.8,2) rectangle ({\i*0.8+0.1},2.1);
\fill ({\i*9/10},2) rectangle ({\i*9/10+0.1},2.1);
}
\foreach \i in {3,4,5} {
\fill ({2.4+\i*0.8+1/10},2) rectangle ({2.4+\i*0.8+2/10},2.1);
\fill ({2.4+\i*9/10},2) rectangle ({2.4+(\i*9+1)/10},2.1);
}
\foreach \i in {2,3,4} {
\pgfmathsetmacro{\start}{\i+1}
\foreach \j in {\start,...,5} {
\fill ({-\i*\i*0.4+3.7*\i+0.8*\j-0.8},2) rectangle ({-\i*\i*0.4+3.7*\i+0.8*\j-0.7},2.1);
\fill ({-\i*\i*0.4+3.6*\i+0.9*\j-0.8},2) rectangle ({-\i*\i*0.4+3.6*\i+0.9*\j-0.7},2.1);
}
}
\foreach \i in {0,...,19} {
\foreach \j in {0,...,5} {
\draw ({\i*0.8+\j/10},3) rectangle ({\i*0.8+(\j+1)/10},3.1);
}
}
\fill (0,3) rectangle (0.3,3.1);
\fill (0.9,3) rectangle (1.2,3.1);
\fill (2*0.8,3) rectangle (1.7,3.1);
\fill (1.8,3) rectangle (2,3.1);
\foreach \i in {3,4,5} {
\fill ({\i*0.8},3) rectangle ({\i*0.8+2/10},3.1);
\fill ({\i*9/10},3) rectangle ({(9*\i+1)/10},3.1);
}
\fill (4.9,3) rectangle (5,3.1);
\fill (5.1,3) rectangle (5.3,3.1);
\fill (5.7,3) rectangle (5.9,3.1);
\fill (6,3) rectangle (6.1,3.1);
\fill (6.5,3) rectangle (6.7,3.1);
\fill (6.9,3) rectangle (7,3.1);
\fill (7.4,3) rectangle (7.7,3.1);
\foreach \i in {2,3} {
\foreach \j in {4,5} {
\fill ({0.8*(2*\i+\j+2)},3) rectangle ({0.8*(2*\i+\j+2)+0.1},3.1);
\fill ({1.7*\i+0.8*\j+1.6},3) rectangle ({1.7*(\i+1)+0.8*\j},3.1);
\fill ({1.6*\i+0.9*\j+1.6},3) rectangle ({1.6*\i+0.9*\j+1.7},3.1);
}
}
\fill (11.2,3) rectangle (11.3,3.1);
\fill (11.6,3) rectangle (11.8,3.1);
\fill (12.1,3) rectangle (12.2,3.1);
\fill (12.3,3) rectangle (12.4,3.1);
\fill (12.5,3) rectangle (12.6,3.1);
\fill (12.9,3) rectangle (13,3.1);
\fill (13.2,3) rectangle (13.4,3.1);
\fill (13.8,3) rectangle (14,3.1);
\fill (14.1,3) rectangle (14.2,3.1);
\fill (14.6,3) rectangle (14.7,3.1);
\fill (14.8,3) rectangle (15,3.1);
\fill (15.5,3) rectangle (15.8,3.1);
\foreach \i in {0,...,14} {
\foreach \j in {0,...,5} {
\draw ({(8*\i+\j)/10},4) rectangle ({(8*\i+\j+1)/10},4.1);
}
}
\fill (0,4) rectangle (0.4,4.1);
\fill (0.9,4) rectangle (1.3,4.1);
\foreach \i in {1,2,3} {
\fill ({0.8*(\i+1)},4) rectangle ({(8+9*\i)/10},4.1);
\fill ({(1+\i)*0.9},4) rectangle ({0.8*\i+1.3},4.1);
}
\fill (4,4) rectangle (4.3,4.1);
\fill (4.5,4) rectangle (4.6,4.1);
\foreach \i in {2,3,4} {
\fill (3.3+\i*0.8,4) rectangle (3.2+\i*0.9,4.1);
\fill (3.3+0.9*\i,4) rectangle (3.8+0.8*\i,4.1);
}
\fill (7.4,4) rectangle (7.8,4.1);
\fill (8,4) rectangle (8.1,4.1);
\fill (8.2,4) rectangle (8.3,4.1);
\fill (8.4,4) rectangle (8.6,4.1);
\fill (8.8,4) rectangle (8.9,4.1);
\fill (9,4) rectangle (9.2,4.1);
\fill (9.3,4) rectangle (9.4,4.1);
\fill (9.6,4) rectangle (9.7,4.1);
\fill (9.9,4) rectangle (10.2,4.1);
\fill (10.4,4) rectangle (10.6,4.1);
\fill (10.7,4) rectangle (10.8,4.1);
\fill (10.9,4) rectangle (11,4.1);
\fill (11.2,4) rectangle (11.4,4.1);
\fill (11.6,4) rectangle (11.8,4.1);
\foreach \i in {0,...,5} {
\foreach \j in {0,...,5} {
\draw ({(8*\i+\j)/10},5) rectangle ({(8*\i+\j+1)/10},5.1);
}
}
\fill (0,5) rectangle (0.5,5.1);
\foreach \i in {1,...,5} {
\fill ({0.9*\i},5) rectangle ({(8*\i+6)/10},5.1);
}
\foreach \i in {2,...,5} {
\fill ({0.8*\i},5) rectangle ({(9*\i-1)/10},5.1);
}
\fill (0,6) rectangle (0.6,6.1);

\draw[dashed,thick] (0.3,0.1) -- +(0,5.9);
\foreach \i in {1,...,5} {
\draw[dashed,thick] (0.8*\i+0.3,1.1) -- +(0,3.9);
}
\foreach \i in {6,...,14} {
\draw[dashed,thick] (0.8*\i+0.3,2.1) -- +(0,1.9);
}
\draw[thick,red] (0.3,0.1) -- (1.1,1) (1.1,1.1) -- (0.3,2) (0.3,2.1) -- (2.7,3) (2.7,3.1) -- (0.3,4) (0.3,4.1) -- (4.3,5) (4.3,5.1) -- (0.3,6);
\draw[thick,green] (0.3,1.1) -- (1.9,2) (1.9,2.1) -- (0.3,3) (0.3,3.1) -- (3.5,4) (3.5,4.1) -- (0.3,5);
\draw[thick,green] (1.9,1.1) -- (7.5,2) (7.5,2.1) -- (1.9,3) (1.9,3.1) -- (9.1,4) (9.1,4.1) -- (1.9,5);
\draw[thick,green] (2.7,1.1) -- (5.1,2) (5.1,2.1) -- (1.1,3) (1.1,3.1) -- (6.7,4) (6.7,4.1) -- (1.1,5);
\draw[thick,green] (3.5,1.1) -- (5.9,2) (5.9,2.1) -- (3.5,3) (3.5,3.1) -- (11.5,4) (11.5,4.1) -- (3.5,5);
\draw[thick,green] (4.3,1.1) -- (6.7,2) (6.7,2.1) -- (4.3,3) (4.3,3.1) -- (10.7,4) (10.7,4.1) -- (2.7,5);
\draw[thick,blue] (2.7,2.1) -- (9.9,3) (9.9,3.1) -- (2.7,4);
\draw[thick,blue] (3.5,2.1) -- (8.3,3) (8.3,3.1) -- (1.9,4);
\draw[thick,blue] (4.3,2.1) -- (9.1,3) (9.1,3.1) -- (4.3,4);
\draw[thick,blue] (1.1,2.1) -- (5.9,3) (5.9,3.1) -- (1.1,4);
\draw[thick,blue] (8.3,2.1) -- (14.7,3) (14.7,3.1) -- (8.3,4);
\draw[thick,blue] (9.1,2.1) -- (13.9,3) (13.9,3.1) -- (7.5,4);
\draw[thick,blue] (9.9,2.1) -- (15.5,3) (15.5,3.1) -- (9.9,4);
\draw[thick,blue] (10.7,2.1) -- (12.3,3) (12.3,3.1) -- (5.1,4);
\draw[thick,blue] (11.5,2.1) -- (13.1,3) (13.1,3.1) -- (5.9,4);
\end{tikzpicture}
}
\caption{The edge-disjoint SCDs~$\cD_0$ (dashed vertical paths) and~$\cD_1$ (solid paths; chains of the same length are drawn with the same color) in the 6-cube.
The bitstrings are drawn with white squares representing 0s and black squares representing~1s.
}
\label{fig:d01}
\end{figure}

\begin{table}
\caption{Known pairwise edge-disjoint SCDs in the $n$-cube for $n=1,2,\ldots,11$.
The definitions of~$\cX_5,\cY_5,\cZ_5$ and~$\cX_7,\cY_7$ are given in Section~\ref{sec:4scds-odd}, and the product operation $\times$ is described in Section~\ref{sec:prod-proof}.
}
\makebox[0cm]{ 
\begin{tabular}{c|lllllllllll}
$n$                    & 1 & 2 & 3 & 4 & 5 & 6 & 7 & 8 & 9 & 10 & 11 \\ \hline
$\lfloor n/2\rfloor+1$ & 1 & 2 & 2 & 3 & 3 & 4 & 4 & 5 & 5 & 6 & 6 \\
SCDs                   & $\cD_0$
                       & $\cD_0,\ol{\cD_0}$
                       & $\cD_0,\ol{\cD_0}$
                       & $\cD_0,\ol{\cD_0},$
                       & $\cX_5,\cY_5,$
                       & $\cD_0,\ol{\cD_0},$
                       & $\cX_7,\ol{\cX_7},$
                       & $\cD_0,\ol{\cD_0},$
                       & $\cD_0(4)\times \cX_5,$
                       & $\cD_0,\ol{\cD_0},$ 
                       & $\cD_0(6)\times \cX_5,$ \\
                       & & & & $\cD_1$ & $\cZ_5$ & $\cD_1,\ol{\cD_1}$ & $\cY_7,\ol{\cY_7}$ & $\cD_1,\ol{\cD_1}$ & $\ol{\cD_0(4)}\times \cY_5,$ & $\cD_1,\ol{\cD_1}$ & $\ol{\cD_0(6)}\times \cY_5,$ \\
                       & & & & & & & & & $\cD_1(4)\times \cZ_5$ & & $\cD_1(6)\times \cZ_5$
\end{tabular}
}
\label{tab:small-scds}
\end{table}

\subsection{Related work}

Apart from building Gray codes, symmetric chain decompositions have many other interesting applications, e.g., to construct rotation-symmetric Venn diagrams for $n$~sets when $n$ is a prime number~\cite{MR2034416, MR2268388}, and to solve the Littlewood-Offord problem on sums of vectors~\cite{MR866142}.
It would be very interesting to investigate how the new SCDs of the $n$-cube presented in this paper can be exploited for those and other applications.

A notion that is closely related to edge-disjoint SCDs is that of \emph{orthogonal} chain decompositions, which were first considered by Shearer and Kleitman~\cite{MR532807}.
Two chain decompositions are called \emph{orthogonal} if every pair of chains has at most one vertex in common, where one also allows chains that are not symmetric around the middle or chains that skip some levels.
Shearer and Kleitman showed that~$\cD_0$ and~$\ol{\cD_0}$ are almost orthogonal (only the longest chains have two elements in common), and they conjectured that the $n$-cube has $\lfloor n/2\rfloor+1$ pairwise orthogonal chain decompositions where each decomposition consists of $\binom{n}{\lfloor n/2\rfloor}$ many chains.
Spink~\cite{spink:17} and Däubel et al.~\cite{MR4017317} recently made some progress towards this conjecture, establishing~4 as the best known lower bound.

Pikhurko~\cite{MR1765729} showed via a parenthesis matching argument that all edges of the $n$-cube can be decomposed into symmetric chains.
However, it is not clear whether these chains contain a subset that forms an SCD.
An interesting construction relating Hamilton cycles and SCDs in the $n$-cube was presented by Streib and Trotter~\cite{MR3268651}.
They inductively construct a Hamilton cycle in the $n$-cube for any~$n\geq 2$ that can be partitioned into symmetric chains forming an SCD.
This Hamilton cycle has the minimal number of `peaks' where the differences in the Hamming weight change sign (and thus also the minimal number of corresponding `valleys').

\subsection{Proof ideas and outlook}
\label{sec:outlook}

Our proofs of Theorem~\ref{thm:gmlc4} and Theorem~\ref{thm:4scds-even} are based on the so-called $i$-lexical matchings introduced by Kierstead and Trotter~\cite{MR962224}, which form a factorization of all edges between the middle two levels of the $(2n+1)$-cube.
We generalize these matchings to arbitrary consecutive levels of the cube of arbitrary dimension, and combine suitable subsets of those matching edges into a cycle factor and symmetric chain decompositions, respectively.
In particular, the four SCDs in Theorem~\ref{thm:4scds-even}, referred to as $\cD_0$, $\cD_1$, $\ol{\cD_0}$ and $\ol{\cD_1}$ before the theorem, are the union of all 0-lexical and 1-lexical matching edges, and of their complements, respectively.
For the proof of Theorem~\ref{thm:gmlc4}, the cycles of the factor obtained as the union of certain lexical matching edges are joined to a Hamilton cycle by local modifications.
As mentioned before, Theorem~\ref{thm:4scds-odd} is derived from Theorem~\ref{thm:prod}, which is proved by a straightforward adaptation of the arguments from~\cite{MR0043115}.

Subsequent to our work, Gregor, Mi\v{c}ka, and M\"utze~\cite{DBLP:conf/icalp/GregorMM20} solved Problem~\ref{prob:gmlc} in full generality.
Their proof generalizes our construction for the case $\ell=2$, and also uses lexical matchings extensively.
Furthermore, D\"aubel, J\"ager, M\"utze and Scheucher~\cite{MR4017317} showed that there are four edge-disjoint SCDs in the $n$-cube for all $n\geq 6$, ruling out the two possible exceptions~$n=9,11$.
Moreover, they proved that there are five edge-disjoint SCDs for all $n\geq 90$, and six edge-disjoint SCDs for all dimensions~$n=11k$, $k\geq 1$.

\subsection{Outline of this paper}

In Section~\ref{sec:defs} we introduce several definitions that will be used throughout this paper.
In Section~\ref{sec:scds} we present the proofs of Theorems~\ref{thm:4scds-even}--\ref{thm:prod}, and we describe the construction of the SCD~$\cD_1$ and of the SCDs in~$Q_5$ and~$Q_7$ referred to in Table~\ref{tab:small-scds}.
As it is somewhat technical, we defer the proof of Theorem~\ref{thm:gmlc4} to Section~\ref{sec:gmlc4}.
In Section~\ref{sec:exp} we present results from computer experiments on the cycle factors through the middle $2\ell$~levels of the $(2n+1)$-cube constructed as in the proof of Theorem~\ref{thm:gmlc-2fac}.
We conclude in Section~\ref{sec:open} with some open problems.

\section{Preliminaries}
\label{sec:defs}

We begin by introducing some terminology that is used throughout the following sections.

\subsection{Bitstrings, lattice paths, and rooted trees}
\label{sec:dyck}

We use~$L_{n,k}$ to denote the set of all bitstrings of length~$n$ with Hamming weight~$k$, so this is exactly the $k$th level of~$Q_n$.
For any bitstring~$x$, we write~$\ol{x}$ for its complement and~$\rev(x)$ for the reversed bitstring.

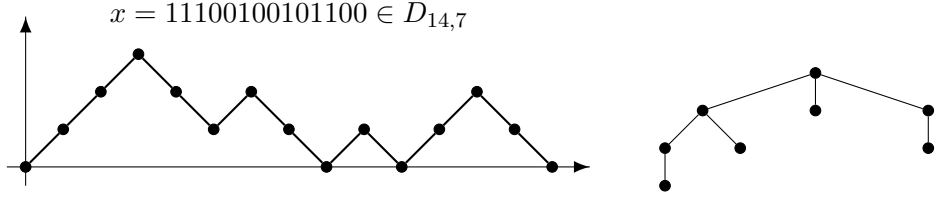
\begin{figure}
\begin{tikzpicture}[scale=0.5]

\draw[ptr] (-0.5,0) -- (15,0);
\draw[ptr] (0,-0.5) -- (0,4);

\node[node_black] (0) at (0,0) {};
\node[node_black] (1) at (1,1) {};
\node[node_black] (2) at (2,2) {};
\node[node_black] (3) at (3,3) {};
\node[node_black] (4) at (4,2) {};
\node[node_black] (5) at (5,1) {};
\node[node_black] (6) at (6,2) {};
\node[node_black] (7) at (7,1) {};
\node[node_black] (8) at (8,0) {};
\node[node_black] (9) at (9,1) {};
\node[node_black] (10) at (10,0) {};
\node[node_black] (11) at (11,1) {};
\node[node_black] (12) at (12,2) {};
\node[node_black] (13) at (13,1) {};
\node[node_black] (14) at (14,0) {};

\path[line_solid] (0) -- (1) -- (2) -- (3) -- (4) -- (5) -- (6) -- (7) -- (8) -- (9) -- (10) -- (11) -- (12) -- (13) -- (14);

\node[node_black] (Z) at (17,-0.5) {};
\node[node_black] (A) at (17,0.5) {} edge (Z);
\node[node_black] (B) at (19,0.5) {};
\node[node_black] (C) at (18,1.5) {} edge (A) edge (B);
\node[node_black] (D) at (21,1.5) {};
\node[node_black] (E) at (24,0.5) {};
\node[node_black] (F) at (24,1.5) {} edge (E);
\node[node_black] at (21,2.5) {} edge (C) edge (D) edge (F);

\node at (7,4) {$x=11100100101100\in D_{14,7}$};

\end{tikzpicture}
\caption{The correspondence between bitstrings, lattice paths (left) and rooted trees (right).}
\label{fig:tree}
\end{figure}

We often interpret a bitstring~$x$ as a path in the integer lattice~$\mathbb{Z}^2$ starting at the origin~$(0,0)$, where every 1-bit is interpreted as an $\upstep$-step that changes the current coordinate by~$(+1,+1)$ and every 0-bit is interpreted as a $\downstep$-step that changes the current coordinate by~$(+1,-1)$; see Figure~\ref{fig:tree}.
Let $D_{n,k}\seq L_{n,k}$ denote the bitstrings that have the property that in every prefix, the number of~1s is at least as large as the number of 0s.
We partition the set~$D_{n,k}$ further into~$D_{n,k}^{>0}$ and~$D_{n,k}^{=0}$, according to whether this inequality is strict for all non-empty prefixes, or whether it holds with equality for at least one non-empty prefix, respectively.
The empty bitstring~$()$ therefore belongs to~$D_{0,0}^{>0}$ and not to~$D_{0,0}^{=0}$.
We also define $D^{>0}:=\bigcup_{n\geq k\geq 0} D_{n,k}^{>0}$, $D^{=0}:=\bigcup_{n\geq k\geq 0} D_{n,k}^{=0}$ and $D:=\bigcup_{k\geq 0} D_{2k,k}$.
In terms of lattice paths, $D$ corresponds to so-called \emph{Dyck paths} that never move below the abscissa~$y=0$ and end at the abscissa.
Similarly, $D^{>0}$ are paths that always stay strictly above the abscissa except at the origin, and~$D^{=0}$ are paths that touch the abscissa at least once more.
Any bitstring~$x\in D^{=0}$ can be written uniquely as $x=(1,u,0,v)$ with~$u\in D$.
We refer to this as the \emph{canonical decomposition} of~$x$.
The set~$D_{n,k}^-$ is defined similarly as~$D_{n,k}$, but we require that in exactly one prefix, the number of~1s is strictly smaller than the number of~0s.
That is, the lattice paths corresponding to~$D_{n,k}^-$ move below the abscissa exactly once.

Note that for any bitstring~$x$, the inverted bitstring~$\ol{x}$ corresponds to mirroring the lattice path at a horizontal line, and the inverted and reversed bitstring~$\ol{\rev}(x)$ corresponds to mirroring the lattice path at a vertical line.
In particular, we have~$\ol{\rev}(x)\in D$ for every~$x\in D$.

An \emph{(ordered) rooted tree} is a tree with a specified root vertex, and the children of each vertex have a specified left-to-right ordering.
We think of a rooted tree as a tree embedded in the plane with the root on top, with downward edges leading from any vertex to its children, and the children appear in the specified left-to-right ordering. 
Using a standard Catalan bijection, every Dyck path~$x\in D_{2n,n}$ can be interpreted as a rooted tree with $n$~edges; see Figure~\ref{fig:tree} and~\cite{MR3467982}.

\subsection{Lexical matchings}
\label{sec:lex}

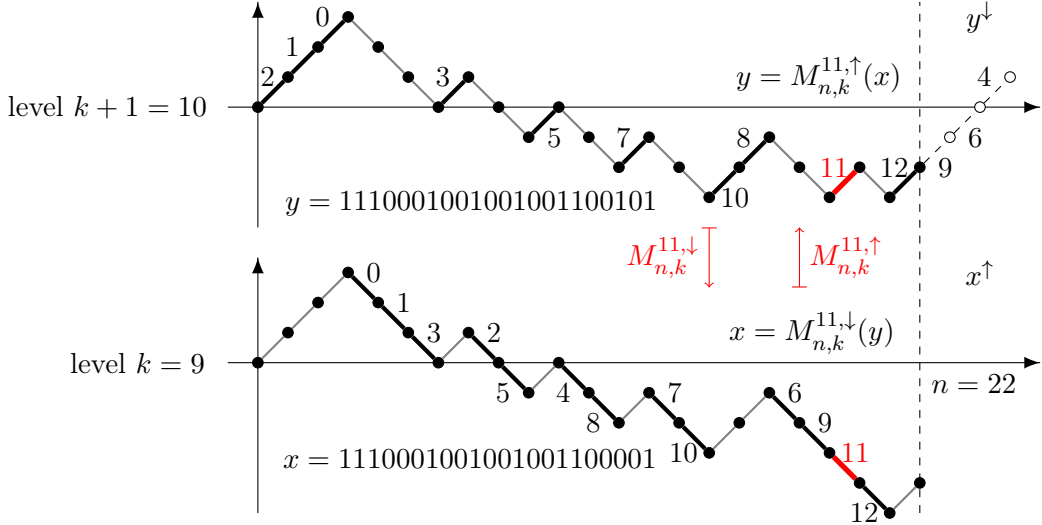
\begin{figure}
\tikzset{dlabela/.style={xshift=1.4mm,yshift=2mm}}
\tikzset{dlabelb/.style={xshift=-1.4mm,yshift=-2mm}}
\tikzset{ulabela/.style={xshift=-1.4mm,yshift=2mm}}
\tikzset{ulabelb/.style={xshift=1.4mm,yshift=-2mm}}

\begin{tikzpicture}[scale=0.4]

\draw[ptr] (-1,0) -- (26,0);
\draw[ptr] (0,-5) -- (0,3.5);

\node[node_black] (0) at (0,0) {};
\node[node_black] (1) at (1,1) {};
\node[node_black] (2) at (2,2) {};
\node[node_black] (3) at (3,3) {};
\node[node_black] (4) at (4,2) {};
\node[node_black] (5) at (5,1) {};
\node[node_black] (6) at (6,0) {};
\node[node_black] (7) at (7,1) {};
\node[node_black] (8) at (8,0) {};
\node[node_black] (9) at (9,-1) {};
\node[node_black] (10) at (10,0) {};
\node[node_black] (11) at (11,-1) {};
\node[node_black] (12) at (12,-2) {};
\node[node_black] (13) at (13,-1) {};
\node[node_black] (14) at (14,-2) {};
\node[node_black] (15) at (15,-3) {};
\node[node_black] (16) at (16,-2) {};
\node[node_black] (17) at (17,-1) {};
\node[node_black] (18) at (18,-2) {};
\node[node_black] (19) at (19,-3) {};
\node[node_black] (20) at (20,-4) {};
\node[node_black] (21) at (21,-5) {};
\node[node_black] (22) at (22,-4) {};

\path[line_solid,gray] (0) -- (1) -- (2) -- (3) (6) -- (7) (9) -- (10) (12) -- (13) (15) -- (16) -- (17) (21) -- (22);
\draw[ultra thick] (3) -- node[dlabela] {0} (4) -- node[dlabela] {1} (5) -- node[dlabela] {3} (6) (7) -- node[dlabela] {2} (8) -- node[dlabelb] {5} (9) (10) -- node[dlabelb] {4} (11) -- node[dlabelb] {8} (12) (13) -- node[dlabela] {7} (14) -- node[dlabelb] {10} (15) (17) -- node[dlabela] {6} (18) -- node[dlabela] {9} (19) (20) -- node[dlabelb] {12} (21);
\draw[line width=2pt, red] (19) -- node[dlabela] {11} (20);

\node[anchor=west] at (0.5,-3.2) {$x=1110001001001001100001$};
\node[anchor=east] at (21.5,1) {$x=M_{n,k}^{11,\downarrow}(y)$};
\node at (24,3) {$x^\uparrow$};
\node at (-4,0) {level $k=9$};
\node at (23.8,-0.7) {$n=22$};

\draw[red,|-to] (18,2.5) -- node[right] {$M_{n,k}^{11,\uparrow}$} (18,4.5);
\draw[red,|-to] (15,4.5) -- node[left] {$M_{n,k}^{11,\downarrow}$} (15,2.5);

\draw[dashed] (22,-5) -- (22,12);

\begin{scope}[yshift=8.5cm]
\draw[ptr] (-1,0) -- (26,0);
\draw[ptr] (0,-4) -- (0,3.5);

\node[node_black] (0) at (0,0) {};
\node[node_black] (1) at (1,1) {};
\node[node_black] (2) at (2,2) {};
\node[node_black] (3) at (3,3) {};
\node[node_black] (4) at (4,2) {};
\node[node_black] (5) at (5,1) {};
\node[node_black] (6) at (6,0) {};
\node[node_black] (7) at (7,1) {};
\node[node_black] (8) at (8,0) {};
\node[node_black] (9) at (9,-1) {};
\node[node_black] (10) at (10,0) {};
\node[node_black] (11) at (11,-1) {};
\node[node_black] (12) at (12,-2) {};
\node[node_black] (13) at (13,-1) {};
\node[node_black] (14) at (14,-2) {};
\node[node_black] (15) at (15,-3) {};
\node[node_black] (16) at (16,-2) {};
\node[node_black] (17) at (17,-1) {};
\node[node_black] (18) at (18,-2) {};
\node[node_black] (19) at (19,-3) {};
\node[node_black] (20) at (20,-2) {};
\node[node_black] (21) at (21,-3) {};
\node[node_black] (22) at (22,-2) {};
\node[node_white] (23) at (23,-1) {};
\node[node_white] (24) at (24,0) {};
\node[node_white] (25) at (25,1) {};

\tikzset{every node/.style={inner sep=1pt}}

\path[line_solid,gray] (3) -- (4) -- (5) -- (6) (7) -- (8) -- (9) (10) -- (11) -- (12) (13) -- (14) -- (15) (17) -- (18) -- (19) (20) -- (21);
\draw[ultra thick] (0) -- node[ulabela,xshift=0.7mm] {2} (1) -- node[ulabela] {1} (2) -- node[ulabela] {0} (3) (6) -- node[ulabela] {3} (7) (9) -- node[ulabelb] {5} (10) (12) -- node[ulabela] {7} (13) (15) -- node[ulabelb] {10} (16) -- node[ulabela] {8} (17) (21) -- node[ulabela] {12} (22);
\draw[line width=2pt, red] (19) -- node[ulabela] {11} (20);
\draw[dashed] (22) -- node[ulabelb] {9} (23) -- node[ulabelb] {6} (24) -- node[ulabela] {4} (25);

\node[anchor=west] at (0.8,-3.2) {$y=1110001001001001100101$};
\node[anchor=east] at (21.5,1) {$y=M_{n,k}^{11,\uparrow}(x)$};
\node at (24,3) {$y^{\downarrow}$};
\node at (-5,0) {level $k+1=10$};
\end{scope}

\end{tikzpicture}
\caption{Definition of $i$-lexical matchings between levels~9 and~10 of~$Q_{22}$, where steps flipped along the $i$-lexical matching edge are marked with~$i$.
Between those two levels, the vertex~$x$ is incident with $i$-lexical matching edges for each $i\in\{0,1,\ldots,12\}$, and the vertex~$y$ is incident with $i$-lexical matching edges for each $i\in\{0,1,\ldots,12\}\setminus\{4,6,9\}$.
}
\label{fig:lexmatch}
\end{figure}

We now define the aforementioned \emph{$i$-lexical matchings} of Kierstead and Trotter~\cite{MR962224}.
Originally, they were defined and analyzed for the graph between the middle two levels of the $(2n+1)$-cube in an attempt to tackle the middle two levels problem, and we begin by generalizing them to the $n$-cube for arbitrary $n$ and an arbitrary pair of consecutive levels~$k$ and~$k+1$.
The parameter~$i$ is an integer~$i\in\{0,1,\ldots,n-1\}$, and these matchings are defined as follows; see Figure~\ref{fig:lexmatch}.
Again we interpret a bitstring~$x$ as a lattice path, and we let~$x^\uparrow$ denote the lattice path that is obtained by appending $\downstep$-steps to~$x$ until the resulting path ends at height~$-1$.
If $x$ ends at a height less than~$-1$, then~$x^\uparrow:=x$.
Similarly, we let~$x^\downarrow$ denote the lattice path obtained by appending $\upstep$-steps to~$x$ until the resulting path ends at height~$+1$.
If $x$ ends at a height greater than~$+1$, then~$x^\downarrow:=x$.
We define the matching by two partial mappings $M_{n,k}^{i,\uparrow}\colon L_{n,k}\rightarrow L_{n,k+1}$ and $M_{n,k}^{i,\downarrow} \colon L_{n,k+1}\rightarrow L_{n,k}$ defined as follows:
For any~$x\in L_{n,k}$ we consider the lattice path~$x^\uparrow$ and scan it row-wise from top to bottom, and from right to left in each row.
The partial mapping~$M_{n,k}^{i,\uparrow}(x)$ is obtained by flipping the $i$th $\downstep$-step encountered in this fashion, where counting starts with $0,1,\ldots$, if this $\downstep$-step exists and is part of~$x$; otherwise $x$ is left unmatched.
Similarly, for any~$x\in L_{n,k+1}$ we consider the lattice path~$x^\downarrow$ and scan it row-wise from top to bottom, and from left to right in each row.
The partial mapping~$M_{n,k}^{i,\downarrow}(x)$ is obtained by flipping the~$i$th~$\upstep$-step encountered in this fashion if this~$\upstep$-step exists and is part of~$x$; otherwise $x$ is left unmatched.
It is straightforward to verify that these two partial mappings are inverse to each other, so they indeed define a matching between levels~$k$ and~$k+1$ of~$Q_n$, which we denote by~$M_{n,k}^i$.

The following properties of lexical matchings are straightforward consequences of these definitions.

\begin{lemma}
\label{lem:lex}
Let $0\leq k\leq n-1$ and $l:=\max\{k,n-k-1\}$. The lexical matchings defined before have the following properties. 
\begin{enumerate}[label=(\roman*)]
\item For every $0\leq i \leq l$, the matching~$M_{n,k}^i$ saturates all vertices in the smaller of the two levels~$k$ and~$k+1$.
\item The matchings~$M_{n,k}^i$, $i=0,1,\ldots,l$, form a partition of all edges of the subgraph of~$Q_n$ between levels~$k$ and~$k+1$.
\item For every $0\leq i \leq l$ we have $\ol{M_{n,k}^i}=M_{n,n-k-1}^{l-i}$ and $\rev(M_{n,k}^i)=M_{n,k}^{l-i}$.
Consequently, we have $\ol{\rev}(M_{n,k}^i)=M_{n,n-k-1}^i$.
\end{enumerate}
\end{lemma}

Property~(i) holds as in the smaller of the two levels~$k$ and~$k+1$, no steps are appended to the lattice paths and the required steps exist when computing the $i$-lexical matching between those levels.
Property~(ii) holds as the vertices in the smaller of the two levels~$k$ and~$k+1$ have degree~$l+1$ and the matchings~$M_{n,k}^i$ for $i=0,1,\ldots,l$ are pairwise disjoint.
Property~(iii) follows from the observation that complementing a bitstring corresponds to mirroring the lattice path at a horizontal line, and reverting a bitstring corresponds to mirroring the lattice path at a horizontal line and at a vertical line.

\section{Pairwise edge-disjoint SCDs}
\label{sec:scds}

We proceed to prove Theorems~\ref{thm:4scds-even}--\ref{thm:prod}.

\subsection{Proof of Theorem~\ref{thm:4scds-even}}

To prove Theorem~\ref{thm:4scds-even}, we first give an equivalent definition of the SCD~$\cD_0$ defined in the introduction via the parenthesis matching approach; recall Figure~\ref{fig:q5}~(a) and Figure~\ref{fig:paren}.

For even~$n\geq 2$, we consider a vertex~$x\in L_{n,n/2}$ in the middle level~$n/2$ of~$Q_n$, and we define the sequence of vertices reached from~$x$ when moving up the corresponding chain, and the sequence of vertices reached when moving down the chain.
For this we consider the lattice path corresponding to the bitstring~$x$.
This lattice path ends at the coordinate~$(n,0)$ as the number of~0s equals the number of~1s.
We now label a subsequence of $\downstep$-steps of this lattice path with integers $j=1,2,\ldots$ according to the following procedure; see the top part of Figure~\ref{fig:d01flip} for an illustration:
\begin{enumerate}[label=(\alph*0)]
\item We place a marker at the rightmost highest point of~$x$ and set~$j:=1$.
\item If the marker is at height~$h\geq 1$, we label the $\downstep$-step starting at the marker with~$j$, and we move the marker to the starting point of the rightmost $\downstep$-step starting at height~$h-1$.
We set $j:=j+1$ and repeat.
\item If the marker is at height~$h=0$, we stop.
\end{enumerate}
Flipping the $\downstep$-steps of~$x$ marked with $1,2,\ldots$ in this order yields the sequence of vertices reached from~$x$ when moving up the chain containing~$x$.
An analogous labeling procedure obtained by interchanging left and right, $\downstep$-steps and $\upstep$-steps, and starting with ending points yields the sequence of vertices reached from~$x$ when moving down this chain.
We denote this chain by~$C_0(x)$.
Observe that $C_0(x)$ is a symmetric chain, as the height of the marker decreases by~1 in each step, so the number of edges we move up from~$x$ equals the number of edges we move down from~$x$.
It is easy to verify that the SCD~$\cD_0$ defined before via the parenthesis matching approach satisfies
\begin{equation*}
  \cD_0=\bigcup\nolimits_{x\in L_{n,n/2}} C_0(x) \enspace.
\end{equation*}

\begin{figure}
\begin{tikzpicture}[scale=0.4,>=latex]

\draw[ptr] (-1,0) -- (24,0);
\draw[ptr] (0,-1) -- (0,7);
\draw (22,-0.2) -- (22,0.2);

\node[marker_blue] (0) at (0,0) {};
\node[marker_blue] (1) at (1,1) {};
\node[marker_blue] (2) at (2,2) {};
\node[marker_blue] (3) at (3,3) {};
\node[marker_blue] (4) at (4,4) {};
\node[marker_blue] (5) at (5,5) {};
\node[node_black] (6) at (6,4) {};
\node[marker_red] (7) at (7,5) {};
\node[node_black] (8) at (8,4) {};
\node[node_black] (9) at (9,3) {};
\node[marker_red] (10) at (10,4) {};
\node[node_black] (11) at (11,3) {};
\node[node_black] (12) at (12,2) {};
\node[node_black] (13) at (13,3) {};
\node[node_black] (14) at (14,2) {};
\node[node_black] (15) at (15,1) {};
\node[node_black] (16) at (16,2) {};
\node[marker_red] (17) at (17,3) {};
\node[marker_red] (18) at (18,2) {};
\node[marker_red] (19) at (19,1) {};
\node[marker_red] (20) at (20,0) {};
\node[node_black] (21) at (21,-1) {};
\node[node_black] (22) at (22,0) {};

\path[style={draw,line width=0.8mm,blue}] (0) to (1);
\path[style={draw,line width=0.8mm,blue}] (1) to (2);
\path[style={draw,line width=0.8mm,blue}] (2) to (3);
\path[style={draw,line width=0.8mm,blue}] (3) to (4);
\path[style={draw,line width=0.8mm,blue}] (4) to (5);
\path[line_solid] (5) to (6);
\path[line_solid] (6) to (7);
\path[style={draw,line width=0.8mm,red}] (7) to (8);
\path[line_solid] (8) to (9);
\path[line_solid] (9) to (10);
\path[style={draw,line width=0.8mm,red}] (10) to (11);
\path[line_solid] (11) to (12);
\path[line_solid] (12) to (13);
\path[line_solid] (13) to (14);
\path[line_solid] (14) to (15);
\path[line_solid] (15) to (16);
\path[line_solid] (16) to (17);
\path[style={draw,line width=0.8mm,red}] (17) to (18);
\path[style={draw,line width=0.8mm,red}] (18) to (19);
\path[style={draw,line width=0.8mm,red}] (19) to (20);
\path[line_solid] (20) to (21);
\path[line_solid] (21) to (22);

\draw [->,red] (7) to [out=10,in=130] (10);
\draw [->,red] (10) to [out=10,in=130] (17);
\draw [->,red] (17) to [out=0,in=90] (18);
\draw [->,red] (18) to [out=0,in=90] (19);
\draw [->,red] (19) to [out=0,in=90] (20);
\draw [->,blue] (5) to [out=180,in=90] (4);
\draw [->,blue] (4) to [out=180,in=90] (3);
\draw [->,blue] (3) to [out=180,in=90] (2);
\draw [->,blue] (2) to [out=180,in=90] (1);
\draw [->,blue] (1) to [out=180,in=90] (0);

\node[anchor=west] at (0.5,6.5) {$x=1111101001001001100001$};
\node[anchor=west] at (21.7,-0.7) {$n$};
\node[anchor=north,red] at (7.3,4.5) {$1$};
\node[anchor=north,red] at (10.3,3.5) {$2$};
\node[anchor=north,red] at (17.3,2.5) {$3$};
\node[anchor=north,red] at (18.3,1.5) {$4$};
\node[anchor=north,red] at (19.3,0.5) {$5$};
\node[anchor=north,blue] at (4.7,4.5) {$1$};
\node[anchor=north,blue] at (3.7,3.5) {$2$};
\node[anchor=north,blue] at (2.7,2.5) {$3$};
\node[anchor=north,blue] at (1.7,1.5) {$4$};
\node[anchor=north,blue] at (0.7,0.5) {$5$};

\begin{scope}[xshift=25cm]
\tikzset{every node/.style={inner sep=0.5pt}}

 \matrix[anchor=west] at (0,3) (m) [matrix of nodes] {
  1&1&1&1&1&0&1&1&0&1&1&0&1&0&0&1&1&1&1&1&0&1 \\
  1&1&1&1&1&0&1&1&0&1&1&0&1&0&0&1&1&1&1&0&0&1 \\
  1&1&1&1&1&0&1&1&0&1&1&0&1&0&0&1&1&1&0&0&0&1 \\
  1&1&1&1&1&0&1&1&0&1&1&0&1&0&0&1&1&0&0&0&0&1 \\
  1&1&1&1&1&0&1&1&0&1&0&0&1&0&0&1&1&0&0&0&0&1 \\
  1&1&1&1&1&0&1&0&0&1&0&0&1&0&0&1&1&0&0&0&0&1 \\
  1&1&1&1&0&0&1&0&0&1&0&0&1&0&0&1&1&0&0&0&0&1 \\
  1&1&1&0&0&0&1&0&0&1&0&0&1&0&0&1&1&0&0&0&0&1 \\
  1&1&0&0&0&0&1&0&0&1&0&0&1&0&0&1&1&0&0&0&0&1 \\
  1&0&0&0&0&0&1&0&0&1&0&0&1&0&0&1&1&0&0&0&0&1 \\
  0&0&0&0&0&0&1&0&0&1&0&0&1&0&0&1&1&0&0&0&0&1 \\
 };
  \matrix[anchor=west] at (m.east) (v) [matrix of nodes] {
   \textcolor{red}{5}\\
   \textcolor{red}{4}\\
   \textcolor{red}{3}\\
   \textcolor{red}{2}\\
   \textcolor{red}{1}\\
   \textcolor{blue}{1}\\
   \textcolor{blue}{2}\\
   \textcolor{blue}{3}\\
   \textcolor{blue}{4}\\
   \textcolor{blue}{5}\\
  };
 
 \node[left,anchor=east] at (m-6-1) {$x={}$};
 
 \begin{scope}[on background layer]
  \fill[blue!50] (m-8-3.north west) rectangle (m-9-3.south east);
  \fill[blue!50] (m-9-2.north west) rectangle (m-10-2.south east);
  \fill[blue!50] (m-10-1.north west) rectangle (m-11-1.south east);
  \fill[blue!50] (m-7-4.north west) rectangle (m-8-4.south east);
  \fill[blue!50] (m-6-5.north west) rectangle (m-7-5.south east);
  \fill[red!50] (m-5-8.north west) rectangle (m-6-8.south east);
  \fill[red!50] (m-4-11.north west) rectangle (m-5-11.south east);
  \fill[red!50] (m-3-18.north west) rectangle (m-4-18.south east);
  \fill[red!50] (m-2-19.north west) rectangle (m-3-19.south east);
  \fill[red!50] (m-1-20.north west) rectangle (m-2-20.south east);
 \end{scope}
 
 \node at (-2,5) {$C_0(x)$};
\end{scope}

\end{tikzpicture}

\begin{tikzpicture}[scale=0.4,>=latex]

\draw[ptr] (-1,0) -- (24,0);
\draw[ptr] (0,-1) -- (0,7);
\draw (22,-0.2) -- (22,0.2);

\node[node_black] (0) at (0,0) {};
\node[marker_blue] (1) at (1,1) {};
\node[node_black] (2) at (2,2) {};
\node[marker_blue] (3) at (3,3) {};
\node[node_black] (4) at (4,4) {};
\node[marker_blue] (5) at (5,5) {};
\node[node_black] (6) at (6,4) {};
\node[marker_red] (7) at (7,5) {};
\node[node_black] (8) at (8,4) {};
\node[node_black] (9) at (9,3) {};
\node[node_black] (10) at (10,4) {};
\node[node_black] (11) at (11,3) {};
\node[node_black] (12) at (12,2) {};
\node[node_black] (13) at (13,3) {};
\node[node_black] (14) at (14,2) {};
\node[node_black] (15) at (15,1) {};
\node[node_black] (16) at (16,2) {};
\node[marker_red] (17) at (17,3) {};
\node[node_black] (18) at (18,2) {};
\node[marker_red] (19) at (19,1) {};
\node[node_black] (20) at (20,0) {};
\node[node_black] (21) at (21,-1) {};
\node[node_black] (22) at (22,0) {};

\path[line_solid] (0) to (1);
\path[style={draw,line width=0.8mm,blue}] (1) to (2);
\path[style={draw,line width=0.8mm,blue}] (2) to (3);
\path[style={draw,line width=0.8mm,blue}] (3) to (4);
\path[style={draw,line width=0.8mm,blue}] (4) to (5);
\path[style={draw,line width=0.8mm,red}] (5) to (6);
\path[style={draw,line width=0.8mm,blue}] (6) to (7);
\path[line_solid] (7) to (8);
\path[line_solid] (8) to (9);
\path[line_solid] (9) to (10);
\path[style={draw,line width=0.8mm,red}] (10) to (11);
\path[line_solid] (11) to (12);
\path[line_solid] (12) to (13);
\path[style={draw,line width=0.8mm,red}] (13) to (14);
\path[line_solid] (14) to (15);
\path[line_solid] (15) to (16);
\path[line_solid] (16) to (17);
\path[style={draw,line width=0.8mm,red}] (17) to (18);
\path[style={draw,line width=0.8mm,red}] (18) to (19);
\path[line_solid] (19) to (20);
\path[line_solid] (20) to (21);
\path[line_solid] (21) to (22);

\draw [->,red] (7) to [out=10,in=130] (17);
\draw [->,red] (17) to [out=0,in=90] (19);
\draw [->,blue] (5) to [out=180,in=90] (3);
\draw [->,blue] (3) to [out=180,in=90] (1);

\node[anchor=west] at (0.5,6.5) {$x=1111101001001001100001$};
\node[anchor=west] at (21.7,-0.7) {$n$};
\node[anchor=north,red] at (5.3,4.5) {$1$};
\node[anchor=north,red] at (10.3,3.5) {$2$};
\node[anchor=north,red] at (13.3,2.5) {$3$};
\node[anchor=north,red] at (17.3,2.5) {$5$};
\node[anchor=north,red] at (18.3,1.5) {$4$};
\node[anchor=north,blue] at (6.7,4.5) {$1$};
\node[anchor=north,blue] at (3.7,3.5) {$2$};
\node[anchor=north,blue] at (4.7,4.5) {$3$};
\node[anchor=north,blue] at (1.7,1.5) {$4$};
\node[anchor=north,blue] at (2.7,2.5) {$5$};

\begin{scope}[xshift=25cm]
\tikzset{every node/.style={inner sep=0.5pt}}
 \matrix[anchor=west] at (0,3) (m) [matrix of nodes] {
  1&1&1&1&1&1&1&0&0&1&1&0&1&1&0&1&1&1&1&0&0&1 \\
  1&1&1&1&1&1&1&0&0&1&1&0&1&1&0&1&1&0&1&0&0&1 \\
  1&1&1&1&1&1&1&0&0&1&1&0&1&1&0&1&1&0&0&0&0&1 \\
  1&1&1&1&1&1&1&0&0&1&1&0&1&0&0&1&1&0&0&0&0&1 \\
  1&1&1&1&1&1&1&0&0&1&0&0&1&0&0&1&1&0&0&0&0&1 \\
  
  1&1&1&1&1&0&1&0&0&1&0&0&1&0&0&1&1&0&0&0&0&1 \\
  
  1&1&1&1&1&0&0&0&0&1&0&0&1&0&0&1&1&0&0&0&0&1 \\
  1&1&1&0&1&0&0&0&0&1&0&0&1&0&0&1&1&0&0&0&0&1 \\
  1&1&1&0&0&0&0&0&0&1&0&0&1&0&0&1&1&0&0&0&0&1 \\
  1&0&1&0&0&0&0&0&0&1&0&0&1&0&0&1&1&0&0&0&0&1 \\
  1&0&0&0&0&0&0&0&0&1&0&0&1&0&0&1&1&0&0&0&0&1 \\
 };
 \matrix[anchor=west] at (m.east) (v) [matrix of nodes] {
   \textcolor{red}{5}\\
   \textcolor{red}{4}\\
   \textcolor{red}{3}\\
   \textcolor{red}{2}\\
   \textcolor{red}{1}\\
   \textcolor{blue}{1}\\
   \textcolor{blue}{2}\\
   \textcolor{blue}{3}\\
   \textcolor{blue}{4}\\
   \textcolor{blue}{5}\\
  };
 
 \node[left,anchor=east] at (m-6-1) {$x={}$};
 
 \begin{scope}[on background layer]
  \fill[blue!50] (m-8-5.north west) rectangle (m-9-5.south east);
  \fill[blue!50] (m-9-2.north west) rectangle (m-10-2.south east);
  \fill[blue!50] (m-10-3.north west) rectangle (m-11-3.south east);
  \fill[blue!50] (m-7-4.north west) rectangle (m-8-4.south east);
  \fill[blue!50] (m-6-7.north west) rectangle (m-7-7.south east);
  \fill[red!50] (m-5-6.north west) rectangle (m-6-6.south east);
  \fill[red!50] (m-4-11.north west) rectangle (m-5-11.south east);
  \fill[red!50] (m-3-14.north west) rectangle (m-4-14.south east);
  \fill[red!50] (m-2-19.north west) rectangle (m-3-19.south east);
  \fill[red!50] (m-1-18.north west) rectangle (m-2-18.south east);
 \end{scope}
 
 \node at (-2,5) {$C_1(x)$};
\end{scope}

\end{tikzpicture}
\caption{The labeling procedures that define the symmetric chains~$C_0(x)$ (top) and~$C_1(x)$ (bottom).
The markers that define the upward and downward steps of the chains are drawn as a square and a diamond, respectively.
The resulting chain~$C_0(x)$ is the same as the one shown in Figure~\ref{fig:paren}.
}
\label{fig:d01flip}
\end{figure}
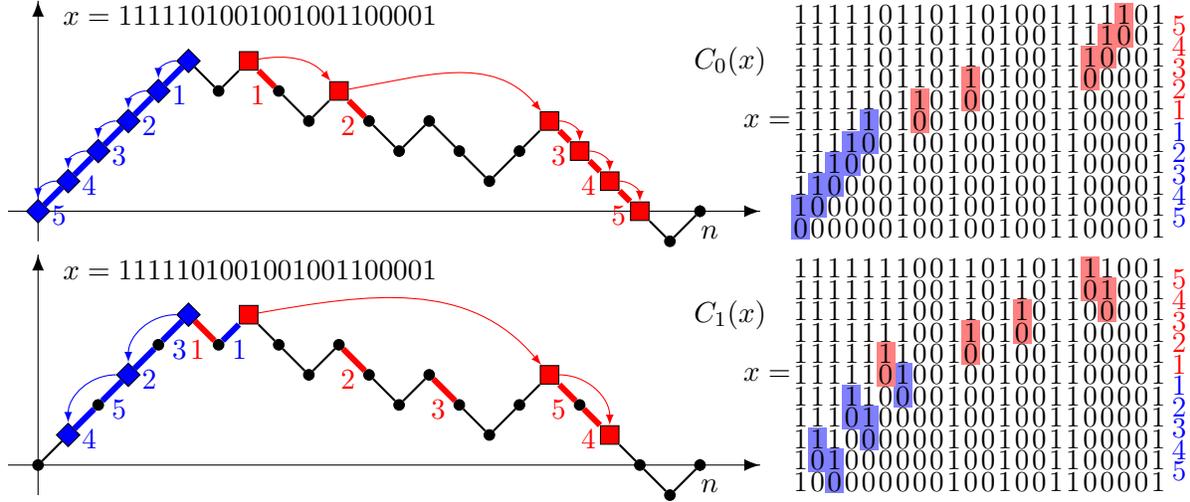

\begin{proof}[Proof of Theorem~\ref{thm:4scds-even}]
We first define a set~$\cD_1$ of chains in~$Q_n$ for even values of~$n\geq 2$ via a labeling rule similar to the rule for~$\cD_0$ described before.
From this definition it follows immediately that all chains in~$\cD_1$ are symmetric.
We then use an equivalent characterization of~$\cD_1$ as the union of certain lexical matchings to show that the chains in~$\cD_1$ form a partition of all vertices of~$Q_n$, proving that $\cD_1$ is an SCD, and that $\cD_0$, $\ol{\cD_0}$, $\cD_1$ and~$\ol{\cD_1}$ are pairwise edge-disjoint.

For even~$n\geq 2$, we consider a vertex~$x \in L_{n,n/2}$ in the middle level~$n/2$ of~$Q_n$.
We interpret it as a lattice path, and label some of its $\downstep$-steps as follows; see the bottom part of Figure~\ref{fig:d01flip}:
\begin{enumerate}[label=(\alph*1)]
\item We place a marker at the rightmost highest point of~$x$ and set~$j:=1$.
If there is a $\downstep$-step to the left of the marker starting at the same height, we label the nearest such step with~$1$ and set~$j:=2$.
\item If the marker is at height~$h\geq 2$, we label the rightmost $\downstep$-step starting at height~$h-1$ with~$j$.
We consider all $\downstep$-steps starting at height~$h-2$ to the right of the labeled step and the $\downstep$-step starting at the marker, we label the second step from the right from this set with~$j+1$, and we move the marker to the starting point of the rightmost $\downstep$-step starting at height~$h-2$.
We set $j:=j+2$ and repeat.
\item If the marker is at height~$h=1$ or~$h=0$, we stop.
\end{enumerate}
We let~$C_1(x)$ denote the chain obtained by flipping bits according to this labeling rule and the corresponding symmetric rule obtained by interchanging left and right, $\downstep$-steps and $\upstep$-steps, and starting with ending points.
Observe that $C_1(x)$ is a symmetric chain, as the height of the marker decreases by~2 in each iteration (and we label two steps in each iteration) and the conditional labeling with~$j=1$ in step~(a1) occurs if and only if the highest point of~$x$ is not unique, so the number of edges we move up from~$x$ equals the number of edges we move down from~$x$.
At this point it is not at all clear yet that the chains~$C_1(x)$, $x\in L_{n,n/2}$, are disjoint, nor that they cover all vertices of~$Q_n$.
This is what we will argue about next, which will prove that
\begin{equation}
\label{eq:def-D1}
  \cD_1:=\bigcup\nolimits_{x\in L_{n,n/2}} C_1(x)
\end{equation}
is actually an SCD of~$Q_n$.

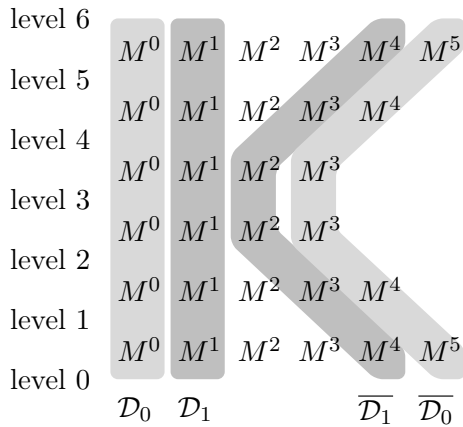
\begin{figure}
\begin{tikzpicture}[scale=0.4,>=latex]

\fill [rounded corners=1mm,fill=gray!30] (2,0) -- (3.8,0) -- (3.8,12) -- (2,12) -- cycle;
\fill [rounded corners=1mm,fill=gray!50] (4,0) -- (5.8,0) -- (5.8,12) -- (4,12) -- cycle;
\fill [rounded corners=1mm,fill=gray!50] (10.8,0) -- (11.8,0) -- (11.8,1) -- (7.5,5) -- (7.5,7) -- (11.8,11) -- (11.8,12) -- (10.8,12) -- (6,7.5) -- (6,4.5) -- cycle;
\fill [rounded corners=1mm,fill=gray!30] (12.8,0) -- (13.8,0) -- (13.8,1) -- (9.5,5) -- (9.5,7) -- (13.8,11) -- (13.8,12) -- (12.8,12) -- (8,7.5) -- (8,4.5) -- cycle;

\node at (0,0) {level $0$};
\node at (0,2) {level $1$};
\node at (0,4) {level $2$};
\node at (0,6) {level $3$};
\node at (0,8) {level $4$};
\node at (0,10) {level $5$};
\node at (0,12) {level $6$};
\foreach \i in {1,3,5,7,9,11}
{
  \node[anchor=west] at (1.9,\i) {$M^0$};
  \node[anchor=west] at (3.9,\i) {$M^1$};
  \node[anchor=west] at (5.9,\i) {$M^2$};
  \node[anchor=west] at (7.9,\i) {$M^3$};
}
\node[anchor=west] at (9.9,1) {$M^4$};
\node[anchor=west] at (9.9,3) {$M^4$};
\node[anchor=west] at (9.9,9) {$M^4$};
\node[anchor=west] at (9.9,11) {$M^4$};
\node[anchor=west] at (11.9,1) {$M^5$};
\node[anchor=west] at (11.9,11) {$M^5$};

\node[anchor=west] at (1.9,-1) {$\cD_0$};
\node[anchor=west] at (3.9,-1) {$\cD_1$};
\node[anchor=west] at (9.9,-1) {$\ol{\cD_1}$};
\node[anchor=west] at (11.9,-1) {$\ol{\cD_0}$};
\end{tikzpicture}
\caption{Unions of lexical matchings~$M^i=M_{n,k}^i$ yielding edge-disjoint chain decompositions in~$Q_n$ for~$n=6$.
The resulting chains in~$\cD_0$ and~$\cD_1$ in~$Q_6$ are shown in Figure~\ref{fig:d01}.
}
\label{fig:d01lex}
\end{figure}

By property~(i) from Lemma~\ref{lem:lex}, for any sequence $\textbf{i}:=(i_0,i_1,\ldots,i_{n-1})$, $i_k\in\{0,1,\ldots,\max\{k,n-k-1\}\}$, the union
\begin{equation}
\label{eq:def-Di}
  \cD_\mathbf{i}:=\bigcup\nolimits_{k=0}^{n-1}M_{n,k}^{i_k}
\end{equation}
is a chain decomposition of~$Q_n$.
The resulting chains are not necessarily symmetric, though.
From the definitions in Section~\ref{sec:lex} it also follows that $\cD_0$ equals the union of the 0-lexical matchings, and that for even~$n\geq 2$, $\cD_1$ as defined in~\eqref{eq:def-D1} equals the union of the 1-lexical matchings; formally we have
\begin{equation*}
  \cD_0=\cD_{(0,0,\ldots,0)}=\bigcup\nolimits_{k=0}^{n-1} M_{n,k}^0 \enspace, \quad 
  \cD_1=\cD_{(1,1,\ldots,1)}=\bigcup\nolimits_{k=0}^{n-1} M_{n,k}^1 \enspace.
\end{equation*}
Consequently, $\cD_1$ is indeed a chain decomposition, and by the definition of~$\cD_1$ via the labeling procedure, all chains in this decomposition are symmetric, so~$\cD_1$ is indeed an SCD.
The fact that $\cD_0$, $\ol{\cD_0}$, $\cD_1$ and~$\ol{\cD_1}$ are pairwise edge-disjoint can be seen by applying property~(iii) from Lemma~\ref{lem:lex} and by observing that by property~(ii), $\cD_\textbf{i}$ and~$\cD_\textbf{j}$ as defined in~\eqref{eq:def-Di} are edge-disjoint if and only if the sequences~$\mathbf{i}$ and~$\mathbf{j}$ differ in every corresponding entry; see Figure~\ref{fig:d01lex}.

This completes the proof.
\end{proof}

Clearly, $\cD_{(0,0,\ldots,0)}$ as defined in~\eqref{eq:def-Di} equals~$\cD_0$ for \emph{every}~$n\geq 1$, so the union of all 0-lexical matchings forms an SCD in any dimension.
In contrast to that, the union of all 1-lexical matchings $\cD_{(1,1,\ldots,1)}$ only forms an SCD for \emph{even}~$n\geq 2$.
Computer experiments show that for~$n\in\{8,10\}$ there is no union of lexical matchings~$\cD_{\mathbf{i}}$ for any sequence~$\mathbf{i}$ that forms an SCD that is edge-disjoint from~$\cD_0$, $\ol{\cD_0}$, $\cD_1$ and~$\ol{\cD_1}$; recall Table~\ref{tab:small-scds}.

Furthermore, taking unions of the so-called modular matchings introduced by Duffus, Kierstead, and Snevily~\cite{MR1268348} does not yield SCDs in~$Q_n$ for~$n=5$ and~$n=7$, and only two edge-disjoint SCDs for~$n=6$.

\subsection{Proof of Theorem~\ref{thm:prod}}
\label{sec:prod-proof}

\begin{figure}
\makebox[0cm]{ 
\begin{tikzpicture}[scale=0.4,>=latex]

\node[node_huge] (b0) at (13,-2) {};
\node[node_huge] (b1) at (9,2) {};
\node[node_huge] (b2) at (13,2) {};
\node[node_huge] (b3) at (17,2) {};
\node[node_huge] (b4) at (9,6) {};
\node[node_huge] (b5) at (13,6) {};
\node[node_huge] (b6) at (17,6) {};
\node[node_huge] (b7) at (13,10) {};
\path[style={draw,line width=1mm}] (b0) to (b1);
\path[style={draw,line width=1mm,dotted}] (b0) to (b3);
\path[style={draw,line width=1mm}] (b1) to (b4);
\path[style={draw,line width=1mm,dotted}] (b1) to (b5);
\path[style={draw,line width=1mm,dotted}] (b2) to (b4);
\path[style={draw,line width=1mm}] (b2) to (b6);
\path[style={draw,line width=1mm}] (b3) to (b5);
\path[style={draw,line width=1mm,dotted}] (b3) to (b6);
\path[style={draw,line width=1mm}] (b4) to (b7);
\path[style={draw,line width=1mm,dotted}] (b6) to (b7);

\foreach \i/\x/\y in {1/2.7/9,2/13/-3,3/9/1,4/13/1,5/17/1,6/9/5,7/13/5,8/17/5,9/13/9} {
  \node[node_black] (\i0) at (\x,\y) {};
  \node[node_black] (\i1) at (\x-1,\y+1) {};
  \node[node_black] (\i2) at (\x+1,\y+1) {};
  \node[node_black] (\i3) at (\x,\y+2) {};
  \path[line_solid] (\i0) to (\i1);
  \path[line_solid] (\i1) to (\i3);
  \path[line_dotted] (\i0) to (\i2);
  \path[line_dotted] (\i2) to (\i3);
}

\node[node_black] (a0) at (7,8) {};
\node[node_black] (a1) at (6,9) {};
\node[node_black] (a2) at (7,9) {};
\node[node_black] (a3) at (8,9) {};
\node[node_black] (a4) at (6,10) {};
\node[node_black] (a5) at (7,10) {};
\node[node_black] (a6) at (8,10) {};
\node[node_black] (a7) at (7,11) {};
\path[line_solid] (a0) to (a1);
\path[line_dotted] (a0) to (a3);
\path[line_solid] (a1) to (a4);
\path[line_dotted] (a1) to (a5);
\path[line_dotted] (a2) to (a4);
\path[line_solid] (a2) to (a6);
\path[line_solid] (a3) to (a5);
\path[line_dotted] (a3) to (a6);
\path[line_solid] (a4) to (a7);
\path[line_dotted] (a6) to (a7);

\path[line_gray_bg] (22,1.5) -- (20,3.5) -- (23,6.5);
\path[line_gray_bg] (23,2.5) -- (22,3.5) -- (24,5.5);
\path[line_gray_bg] (24,3.5) -- (25,4.5);
\path[line_gray_bg] (26,2.5) -- (29,5.5);
\path[line_gray_bg] (32,2.5) -- (30,4.5) -- (31,5.5);
\path[line_gray_bg] (33,3.5) -- (32,4.5);
\path[line_gray_bg] (36,2.5) -- (34,4.5) -- (35,5.5);
\path[line_gray_bg] (37,3.5) -- (36,4.5);
\path[line_gray_bg] (38,3.5) -- (39,4.5);
\path[line_gray_bg] (40,3.5) -- (41,4.5);
\foreach \i in {1,2,3} {
\foreach \j in {1,2,3,4} {
  \node[node_black] (c\i\j) at (22-\i+\j,-0.5+\i+\j) {};
}
}
\path[line_solid] (c11) to (c31);
\path[line_solid] (c12) to (c32);
\path[line_solid] (c13) to (c33);
\path[line_solid] (c14) to (c34);
\path[line_solid] (c11) to (c14);
\path[line_solid] (c21) to (c24);
\path[line_solid] (c31) to (c34);

\foreach \j in {1,2,3,4} {
  \node[node_black] (d\j) at (25+\j,1.5+\j) {};
}
\path[line_solid] (d1) to (d4);

\foreach \i in {1,2,3} {
\foreach \j in {1,2} {
  \node[node_black] (e\i\j) at (32-\i+\j,0.5+\i+\j) {};
}
}
\path[line_solid] (e11) to (e31);
\path[line_solid] (e12) to (e32);
\path[line_solid] (e11) to (e12);
\path[line_solid] (e21) to (e22);
\path[line_solid] (e31) to (e32);

\foreach \i in {1,2,3} {
\foreach \j in {1,2} {
  \node[node_black] (f\i\j) at (36-\i+\j,0.5+\i+\j) {};
}
}
\path[line_solid] (f11) to (f31);
\path[line_solid] (f12) to (f32);
\path[line_solid] (f11) to (f12);
\path[line_solid] (f21) to (f22);
\path[line_solid] (f31) to (f32);

\node[node_black] (g1) at (38,3.5) {};
\node[node_black] (g2) at (39,4.5) {};
\path[line_solid] (g1) to (g2);

\node[node_black] (h1) at (40,3.5) {};
\node[node_black] (h2) at (41,4.5) {};
\path[line_solid] (h1) to (h2);

\path[line_solid] (2.6,7.5) to (5,7.5);
\path[line_dotted] (2.6,6.5) to (5,6.5);

\node[anchor=west] at (1.5,12.3) {$Q_a$};
\node[anchor=west] at (6,12.3) {$Q_b$};
\node[anchor=west] at (11,12.3) {$Q_{a+b}=Q_a\times Q_b$};
\node[anchor=west] at (0.8,11) {$A$};
\node[anchor=west] at (0.8,7.5) {$\cA_1$};
\node[anchor=west] at (0.8,6.5) {$\cA_2$};
\node[anchor=west] at (5.2,11) {$B$};
\node[anchor=west] at (5,7.5) {$\cB_1$};
\node[anchor=west] at (5,6.5) {$\cB_2$};
\node[anchor=west] at (21,8.8) {$A\times B$};
\node[anchor=west] at (29,8.3) {$\cC_1$};
\node[anchor=west] at (19.2,4.7) {$C_1$};
\node[anchor=west] at (20.3,1) {$(x_1,y_1)$};
\node[anchor=west] at (24,5.2) {$(x_1,y_\beta)$};
\node[anchor=west] at (21.3,7.3) {$(x_\alpha,y_\beta)$};

\end{tikzpicture}
}
\caption{Illustration of the proof of Theorem~\ref{thm:prod}.
Construction of two edge-disjoint SCDs in~$Q_5$ from two edge-disjoint SCDs in~$Q_2$ and two edge-disjoint SCDs in~$Q_3$.
The chains of the SCD~$\cC_1$ of~$Q_5$ as constructed in the proof are highlighted in gray.
}
\label{fig:prod}
\end{figure}
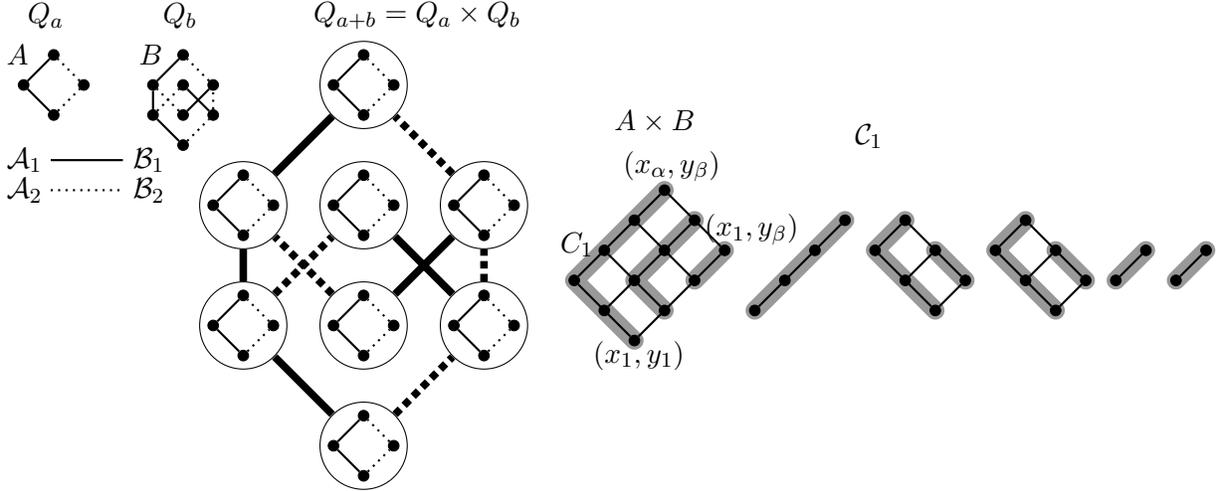

\begin{proof}[Proof of Theorem~\ref{thm:prod}]
For the reader's convenience, this proof is illustrated in Figure~\ref{fig:prod}.
Assume that $\cA_1,\cA_2,\ldots,\cA_k$ and $\cB_1,\cB_2,\ldots,\cB_k$ denote $k$ pairwise edge-disjoint SCDs of~$Q_a$ and~$Q_b$, respectively.
We will think of~$Q_{a+b}$ as the Cartesian product~$Q_a\times Q_b$ of~$Q_a$ and~$Q_b$.
We show how to construct for every~$i\in [k]$ an SCD~$\cC_i$ of $Q_{a+b}=Q_a \times Q_b$ which uses only edges of the form $((u,v),(u',v'))$ where $(u,u')$ is an edge from~$\cA_i$ or~$(v,v')$ is an edge from~$\cB_i$.
From this it follows that the SCDs $\cC_1,\cC_2,\ldots,\cC_k$ are pairwise edge-disjoint.

The SCD~$\cC_i$ of~$Q_{a+b}$ is defined as follows:
The Cartesian products~$A\times B$ of chains~$A\in\cA_i$ and~$B\in\cB_i$ partition the vertices of~$Q_{a+b}$ into two-dimensional grids.
$\cC_i$ is obtained by partitioning each of those grids into symmetric chains in the natural way; see Figure~\ref{fig:prod} (cf.~\cite{MR0043115}):
Specifically, let $A=:(x_1,\ldots,x_\alpha)$ and $B=:(y_1,\ldots,y_\beta)$ be the vertices in the chains~$A$ and~$B$ from bottom to top.
As~$A$ and~$B$ are symmetric, we know that $|x_1|+|x_\alpha|=a$ and $|y_1|+|y_\beta|=b$, where $|x|$ denotes the Hamming weight of the bitstring~$x$.
This implies that $|(x_1,y_1)|+|(x_\alpha,y_\beta)|=|x_1|+|y_1|+|x_\alpha|+|y_\beta|=a+b$, i.e., the bottom and top vertex of the grid~$A\times B$ are on symmetric levels in~$Q_{a+b}$.
We may therefore decompose~$A\times B$ into disjoint symmetric chains~$C_j$, $j=1,2,\ldots,\min\{\alpha,\beta\}$, by setting
\begin{equation}
\label{eq:Cj0}
  C_j := \big((x_1,y_j),(x_2,y_j),\ldots,(x_{\alpha-j+1},y_j),(x_{\alpha-j+1},y_{j+1}),\ldots,(x_{\alpha-j+1},y_\beta)\big) \enspace.
\end{equation}
\end{proof}

Note that in the proof of Theorem~\ref{thm:prod} we have some degrees of freedom in partitioning the two-dimensional grids~$A\times B$ into symmetric chains.
If we perform this construction using $Q_{n+1}=Q_n\times Q_1$ for $n=1,2,\ldots$ and always partition the grids according to~\eqref{eq:Cj0}, then the resulting SCD equals~$\cD_0$.
If instead we partition always according to
\begin{equation}
\label{eq:Cj0p}
  C_j := \big((x_j,y_1),(x_j,y_2),\ldots,(x_j,y_{\beta-j+1}),(x_{j+1},y_{\beta-j+1}),\ldots,(x_\alpha,y_{\beta-j+1})\big) \enspace,
\end{equation}
then we get the SCD~$\ol{\cD_0}$.
The difference between~\eqref{eq:Cj0} and~\eqref{eq:Cj0p} is whether in building~$C_j$ we first move along the first coordinates, or first along the last coordinates.

\subsection{Proof of Theorem~\ref{thm:4scds-odd}}
\label{sec:4scds-odd}

We begin by constructing the SCDs in~$Q_5$ and~$Q_7$ mentioned in Table~\ref{tab:small-scds}.

\begin{lemma}
\label{lem:q57}
$Q_5$ contains three pairwise edge-disjoint SCDs, $Q_7$ contains four pairwise edge-disjoint SCDs, and this is best possible.
\end{lemma}

\begin{figure}[b!]
\makebox[0cm]{ 
\includegraphics[scale=0.916]{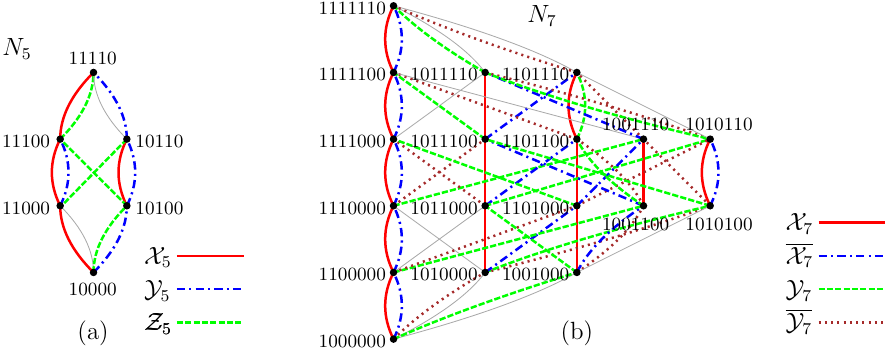}
}
\caption{Illustration of the three edge-disjoint SCDs in~$N_5$ (a) and four edge-disjoint SCDs in~$N_7$ (b).
The names of the SCDs correspond to the ones used in Table~\ref{tab:small-scds}.
In~$N_7$, two pairs of SCDs are complementary.
}
\label{fig:n57}
\end{figure}

\begin{proof}
For prime~$n$, we consider the graph~$Q_n$ with the two vertices in the outermost levels~0 and~$n$ removed, and we identify all bitstrings that differ only by rotation into so-called \emph{necklaces}.
The resulting graph~$N_n$ is a multigraph version of the cover graph of the \emph{necklace poset}.
Specifically, the multiplicity of the edges in~$N_n$ corresponds to the number of ways a bit from a necklace can be flipped to reach the corresponding adjacent necklace.
Here we use that $n$ is a prime number, so every necklace in~$N_n$ corresponds to exactly $n$ bitstrings in~$Q_n$.
For example, in~$N_5$ the necklace~$x:=10000$ has two edges leading to~$y:=11000$, as we can flip the second or the fifth bit in~$x$ to reach~$y$.
This way, every necklace on level~$k$ has $n-k$~edges going up, and $k$~edges going down, like the vertices in~$Q_n$.
The multigraphs~$N_5$ and~$N_7$ are shown in Figure~\ref{fig:n57}.
Now observe that every SCD in~$N_n$ corresponds to an SCD in~$Q_n$, by turning each chain from~$N_n$ into $n$~chains in~$Q_n$ obtained by rotating a representative of each necklace in all possible ways.
Moreover, one of the chains of length~$n-2$ needs to be extended by the all-zero and all-one bitstring to a chain of length~$n$ in~$Q_n$.
Observe that in this way, $k$ edge-disjoint SCDs in~$N_n$ give rise to $k$ edge-disjoint SCDs in~$Q_n$.

As~$n=5$ and~$n=7$ are prime, we thus obtain three edge-disjoint SCDs in~$Q_5$ from the SCDs in~$N_5$ shown in Figure~\ref{fig:n57}~(a), and four edge-disjoint SCDs in~$Q_7$ from the SCDs in~$N_7$ shown in Figure~\ref{fig:n57}~(b).
These SCDs use up all middle edges, so this is best possible.
\end{proof}

\begin{proof}[Proof of Theorem~\ref{thm:4scds-odd}]
For~$n=7$ the statement follows from Lemma~\ref{lem:q57}.
For odd~$n\geq 13$ we apply Theorem~\ref{thm:prod} to~$Q_{n-7}$ and~$Q_7$, using the four edge-disjoint SCDs in~$Q_{n-7}$ given by Theorem~\ref{thm:4scds-even} (note that $n-7\geq 6$), and the four edge-disjoint SCDs in~$Q_7$ given by Lemma~\ref{lem:q57}.
\end{proof}

\section{The middle four levels problem}
\label{sec:gmlc4}

In this section we prove Theorem~\ref{thm:gmlc4}.
The proof proceeds similarly as the proof of the middle two levels problem~\cite{MR3483129, gregor-muetze-nummenpalo:18}.
First, we construct a cycle factor~$\cC_{2n+1}$ in the middle four levels of~$Q_{2n+1}$, and we then modify the cycles in the factor locally to join them to a Hamilton cycle.
Specifically, to join two cycles~$C$ and~$C'$ from our cycle factor, we consider a suitable 6-cycle and take the symmetric difference between their edge sets, so that the result is a single cycle on the union of the vertex sets of~$C$ and~$C'$; see Figure~\ref{fig:c6xy}.
This process is iterated until all cycles are joined to a single Hamilton cycle.
This technique reduces the problem of proving that the middle four levels of~$Q_{2n+1}$ have a Hamilton cycle to the problem of proving that a suitably defined auxiliary graph~$\cH_{n+1}$ has a spanning tree, which is much easier.
This section is organized as follows: We first define the cycle factor~$\cC_{2n+1}$ and analyze its structure.
We then introduce the 6-cycles for the joining operations, and finally show that they can be used to join the cycles of the factor to a Hamilton cycle in the desired fashion.

\subsection{Construction of the cycle factor~\texorpdfstring{$\cC_{2n+1}$}{C2n+1}}
\label{sec:2factor4}

As we are not able to analyze the cycle factor in the middle four levels of~$Q_{2n+1}$ arising from the proof of Theorem~\ref{thm:gmlc-2fac} (see Table~\ref{tab:2factors1} in Section~\ref{sec:exp}), we start with a different construction.
To construct the cycle factor~$\cC_{2n+1}$ in the middle four levels of~$Q_{2n+1}$, we use the $i$-lexical matchings defined in Section~\ref{sec:lex}.
Specifically, we take the union of all $n$-lexical and $(n+1)$-lexical matching edges between the upper two levels~$n+1$ and~$n+2$ and between the lower two levels~$n-1$ and~$n$, plus certain carefully chosen edges~$E$ from the $(n-2)$-lexical, $(n-1)$-lexical and $n$-lexical matching between the middle levels~$n$ and~$n+1$.
Formally, we set
\begin{equation}
\label{eq:def-C}
  \cC_{2n+1}:=(M_{2n+1,n+1}^n \cup M_{2n+1,n+1}^{n+1}) \; \cup \; (M_{2n+1,n-1}^n \cup M_{2n+1,n-1}^{n+1}) \; \cup E \enspace,
\end{equation}
where the set of edges~$E$ is defined in~\eqref{eq:def-E} below.
By this definition and by Lemma~\ref{lem:lex}~(i) and (ii), all vertices in the outer levels~$n-1$ and~$n+2$ have degree two in the subgraph~$\cC_{2n+1}$, and we will choose~$E$ so that all vertices in the inner levels~$n$ and~$n+1$ have degree two as well.

To define the set~$E$, we consider the union of the matchings between the upper two levels
\begin{equation}
\label{eq:def-P}
  \cP:=M_{2n+1,n+1}^n \cup M_{2n+1,n+1}^{n+1} \enspace.
\end{equation}
In the following, for a set of bitstrings~$X$ and a bitstring~$x$, we write~$X\circ x$ for the set obtained by concatenating each bitstring from~$X$ with~$x$.
The set $\cP\circ 00$ is a set of edges in the middle two levels of the $(2n+3)$-cube, and from~\cite[Proposition~2~(i)+(ii)+(iv)]{gregor-muetze-nummenpalo:18}, which talks about a superset of the paths~$\cP\circ 0$, we obtain the following properties.

\begin{lemma}
\label{lem:paths}
For any~$n\geq 1$, the set~$\cP$ defined in~\eqref{eq:def-P} is a set of paths (without any cycles), and the sets of first and last vertices of these paths are $D_{2n+1,n+1}^{=0}$ and $D_{2n+1,n+1}^-$, respectively.
Furthermore, for any path with first vertex $x\in D_{2n+1,n+1}^{=0}$ and last vertex $y\in D_{2n+1,n+1}^-$, if $x=(1,u,0,v)$ is the canonical decomposition of~$x$, then $y=(u,0,1,v)$.
\end{lemma}

The first vertices of the paths~$\cP$, denoted by~$F(\cP)$, are the vertices from level~$n+1$ covered by the matching~$M_{2n+1,n+1}^n$ and not by~$M_{2n+1,n+1}^{n+1}$.
Similarly, the last vertices of the paths~$\cP$, denoted by~$L(\cP)$, are covered by the latter matching but not by the former.
We let~$I(\cP)$ denote the set of vertices in level~$n+1$ covered by neither of the two matchings.
We refer to those vertices as \emph{isolated}.

We let $f:=\ol{\rev}$ denote the automorphism of~$Q_{2n+1}$ that flips all bits and reverses them, i.e., $f(x_1,x_2,\dots,x_{2n+1})=(\ol{x_{2n+1}},\dots,\ol{x_2},\ol{x_1})$.
Using the abbreviation~\eqref{eq:def-P} and Lemma~\ref{lem:lex}~(iii), we may rewrite the definition~\eqref{eq:def-C} equivalently as
\begin{equation}
\label{eq:def-C-eq}
  \cC_{2n+1}:=\cP \cup f(\cP) \cup E \enspace.
\end{equation}

For any set~$X$ of bitstrings and bitstrings~$a,b$ we let~$_a{X_b}$ denote the subset of the bitstrings from~$X$ that have the prefix~$a$ and the suffix~$b$.
Furthermore, let $\cP_0$ and~$\cP_1$ be the collections of paths from~$\cP$ with fixed last bit equal to~0 or~1, respectively.
Note that $\cP=\cP_0\cup \cP_1$, as neither the $n$-lexical matching nor the $(n+1)$-lexical matching between levels~$n+1$ and~$n+2$ uses any edges along which the last bit is flipped.
We start with the following observations.

\begin{lemma}
\label{lem:FLI}
For any~$n\geq 1$, the sets of first, last and isolated vertices of the paths~$\cP$ defined in~\eqref{eq:def-P} are given by
\begin{enumerate}[label=(\roman*)]
\item $F(\cP)=D_{2n+1,n+1}^{=0}$\,, \; $L(\cP)=D_{2n+1,n+1}^-$\,, \; $I(\cP)=D_{2n+1,n+1}^{>0}$\,,
\item $F(\cP_0)=D_{2n,n+1}^{=0}\circ 0$\,, \; $F(\cP_1)=D_{2n,n}^{=0}\circ 1$\,.
\end{enumerate}
\end{lemma}

\begin{proof}
The first two statements in part~(i) are given by Lemma~\ref{lem:paths}.
The statement $I(\cP)=D_{2n+1,n+1}^{>0}$ follows by considering for which $x\in L_{2n+1,n+1}$ the lattice path~$x^{\uparrow}$ (recall the definition from Section~\ref{sec:lex}) has its $n$th and $(n+1)$th $\downstep$-step (in the counting from top to bottom and from right to left in each row, starting from~$0$) in the two $\downstep$-steps that were added to~$x$ at positions~$2n+2$ and~$2n+3$.
One can easily observe that this happens if and only if $x\in D_{2n+1,n+1}^{>0}$.
Part~(ii) follows immediately from the definitions and from part~(i) by considering the lattice paths before the fixed last bit.
\end{proof}

\begin{figure}
\includegraphics[scale=0.916]{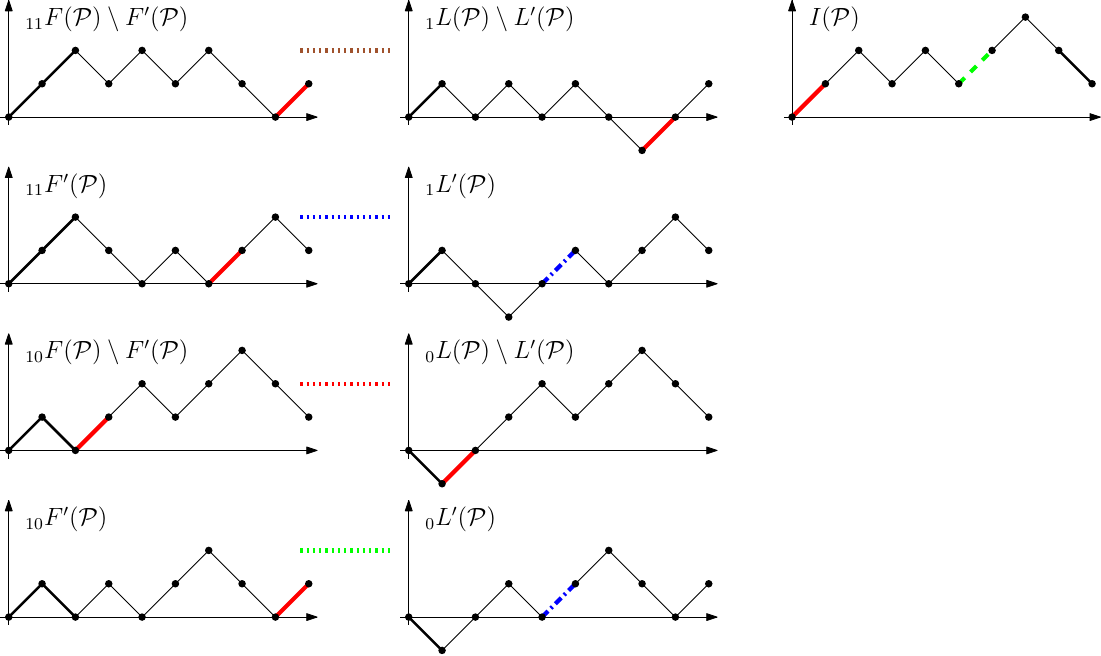}
\caption{Lattice path interpretation of vertex sets in level~$n+1$ of~$Q_{2n+1}$ involved in the construction.
Heavier (black) steps denote fixed bits~0 or~1.
Fat steps are flipped along the edges of~$E$: solid (red) by~$E^n$, dashed (green) by~$E^{n-1}$, dash-dotted (blue) by~$E^{n-2}$.}
\label{fig:dyck-paths}
\end{figure}

We define $F'(\cP) \seq F(\cP)$ as the set of vertices $x \in F(\cP)=D_{2n+1,n+1}^{=0}$ such that $x=(1,u,0,v)$ with some~$u\in D$ and~$v\in D^{=0}$ (otherwise $v\in D^{>0}$).
Similarly, $L'(\cP)\seq L(\cP)$ is the set of all vertices $y \in L(\cP)=D_{2n+1,n+1}^-$ such that $y=(u,0,1,v)$ with some~$u\in D$ and $v\in D^{=0}$.
Figure~\ref{fig:dyck-paths} illustrates the lattice paths corresponding to the different subsets of vertices given by Lemma~\ref{lem:FLI}, refined according to $F'(\cP)$ and $L'(\cP)$ and by fixing the first one or two bits, and which pairs of first and last vertices are joined along paths from~$\cP$ as given by Lemma~\ref{lem:paths}.

We now define the set~$E$ of edges for the cycle factor~$\cC_{2n+1}$ between levels~$n$ and~$n+1$ of~$Q_{2n+1}$ so that each vertex of~$F(\cP)$ and~$L(\cP)$ will be incident with exactly one edge of~$E$, each vertex of~$I(P)$ will be incident with exactly two edges of~$E$, and all other vertices in level~$n+1$ are not incident with any edges from~$E$; see Figure~\ref{fig:4levels}.
Effectively, adding the edges from~$E$ makes all degrees in the subgraph~$\cC_{2n+1}$ in level~$n+1$ equal to two.
We then show that adding the edges from~$E$ also makes all degrees in level~$n$ equal to two, so that $\cC_{2n+1}$ is indeed a cycle factor.

For a set of edges~$M$ and a set of vertices~$X$, we let $M[X]$ denote the set of edges from~$M$ incident with~$X$.
We then define
\begin{subequations}
\label{eq:def-E}
\begin{equation}
  E:=E^n\cup E^{n-1} \cup E^{n-2} \enspace,
\end{equation}
where
\begin{align}
\label{eq:def-En}
  E^n &:= M_{2n+1,n}^{n}[F(\cP) \cup \big(L(\cP) \setminus L'(\cP)\big)\cup  I(\cP)] \enspace, \\
  E^{n-1} &:= M_{2n+1,n}^{n-1}[I(\cP)] \enspace, \\
  E^{n-2} &:= M_{2n+1,n}^{n-2}[L'(\cP)] \enspace.
\end{align}
\end{subequations}

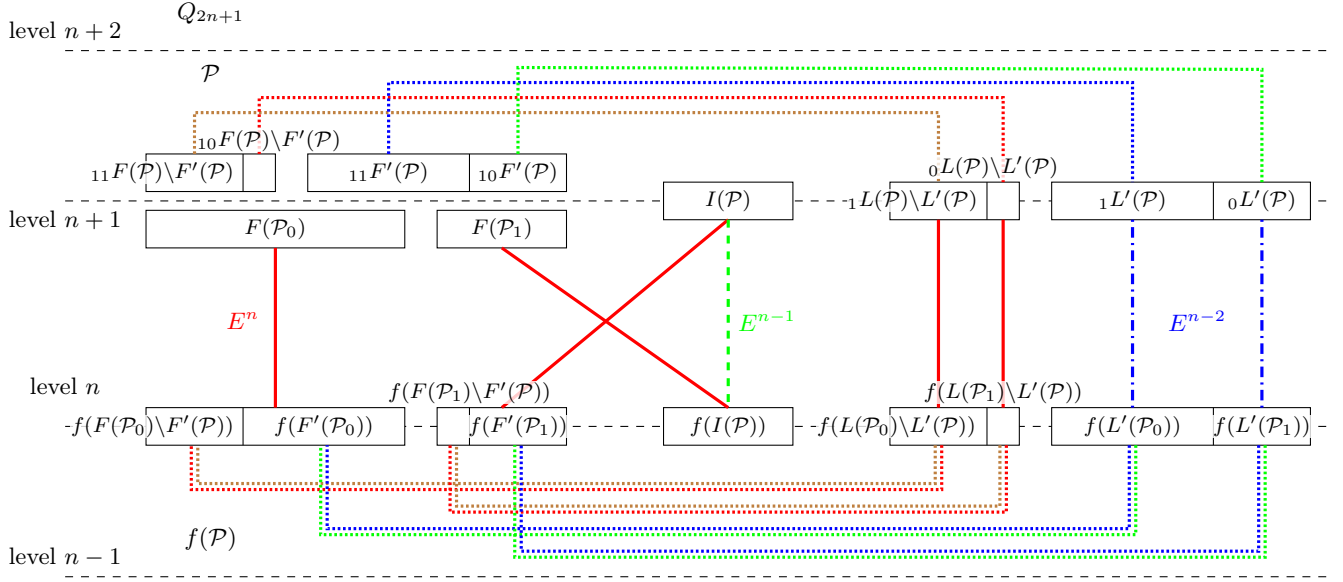
\begin{figure}
\makebox[0cm]{ 
\begin{tikzpicture}[every node/.style={font=\scriptsize},xscale=0.43]

\draw[dashed] (-2.5,5.5) node[above] {\footnotesize level $n+2$} -- (36.5,5.5);
\draw[dashed] (-2.5,3.5) node[below] {\footnotesize level $n+1$} -- (36.5,3.5);
\draw[dashed] (-2.5,0.5) node[above=3mm] {\footnotesize level $n$} -- (36.5,0.5);
\draw[dashed] (-2.5,-1.5) node[above] {\footnotesize level $n-1$} -- (36.5,-1.5);

\draw[fill=white] (0,0.25) rectangle (8,0.75);
\draw[fill=white] (0,2.875) rectangle (8,3.375);
\draw[fill=white] (0,3.625) rectangle (4,4.125);
\draw[fill=white] (9,0.25) rectangle (13,0.75);
\draw[fill=white] (9,2.875) rectangle (13,3.375);
\draw[fill=white] (5,3.625) rectangle (13,4.125);
\draw[fill=white] (16,0.25) rectangle (20,0.75);
\draw[fill=white] (16,3.25) rectangle (20,3.75);
\draw[fill=white] (23,0.25) rectangle (27,0.75);
\draw[fill=white] (23,3.25) rectangle (27,3.75);
\draw[fill=white] (28,0.25) rectangle (36,0.75);
\draw[fill=white] (28,3.25) rectangle (36,3.75);
\foreach \i in {26,33}
 \draw (\i,0.25) -- (\i,0.75) (\i,3.25) -- (\i,3.75);
\draw (3,3.625) -- (3,4.125);
\draw (10,3.625) -- (10,4.125);
\foreach \i in {3,10}
 \draw (\i,0.25) -- (\i,0.75);

\draw[brown,very thick,dotdot] (24.4,0.25) -- (24.4,-0.26) -- (1.6,-0.26) -- (1.6,0.25);
\draw[red,very thick,dotdot] (24.6,0.25) -- (24.6,-0.34) -- (1.4,-0.34) -- (1.4,0.25);
\draw[brown,very thick,dotdot] (26.4,0.25) -- (26.4,-0.56) -- (9.6,-0.56) -- (9.6,0.25);
\draw[red,very thick,dotdot] (26.6,0.25) -- (26.6,-0.64) -- (9.4,-0.64) -- (9.4,0.25);
\draw[blue,very thick,dotdot] (30.4,0.25) -- (30.4,-0.86) -- (5.6,-0.86) -- (5.6,0.25);
\draw[green,very thick,dotdot] (30.6,0.25) -- (30.6,-0.94) -- (5.4,-0.94) -- (5.4,0.25);
\draw[blue,very thick,dotdot] (34.4,0.25) -- (34.4,-1.16) -- (11.6,-1.16) -- (11.6,0.25);
\draw[green,very thick,dotdot] (34.6,0.25) -- (34.6,-1.24) -- (11.4,-1.24) -- (11.4,0.25);

\draw[brown,very thick,dotdot] (1.5,4.125) -- (1.5,4.675) -- (24.5,4.675) -- (24.5,3.75);
\draw[red,very thick,dotdot] (3.5,4.125) -- (3.5,4.875) -- (26.5,4.875) -- (26.5,3.75);
\draw[blue,very thick,dotdot] (7.5,4.125) -- (7.5,5.075) -- (30.5,5.075) -- (30.5,3.75);
\draw[green,very thick,dotdot] (11.5,4.125) -- (11.5,5.275) -- (34.5,5.25) -- (34.5,3.75);

\draw[very thick,red] (4,0.75) -- (4,2.875);
\draw[very thick,red] (11,0.75) -- (18,3.25);
\draw[very thick,red] (18,0.75) -- (11,2.875);
\draw[very thick,red] (24.5,3.25) -- (24.5,0.75);
\draw[very thick,red] (26.5,3.25) -- (26.5,0.75);

\draw[green,very thick,style=dashed] (18,3.25) -- (18,0.75);

\draw[blue,very thick,style=dashdot] (30.5,3.25) -- (30.5,0.75);
\draw[blue,very thick,style=dashdot] (34.5,3.25) -- (34.5,0.75);

\node[fill=white,fill opacity=0.8,text opacity=1,inner sep=0.2pt] at (0.2,0.5) {$f(F(\cP_0) {\setminus} F'(\cP))$};
\node (1fF') at (5.5,0.5) {$f(F'(\cP_0))$};
\node (0fF-) at (9.5,0.5) {};
\node[fill=white,fill opacity=0.8,text opacity=1,inner sep=0.2pt] at (10,0.95) {$f(F(\cP_1){\setminus}F'(\cP))$};
\node[fill=white,fill opacity=0.8,text opacity=1,inner sep=0.2pt] at (11.5,0.5) {$f(F'(\cP_1))$};

\node[fill=white,fill opacity=0.8,text opacity=1,inner sep=0.2pt] at (23.2,0.5) {$f(L(\cP_0){\setminus} L'(\cP))$};
\node[fill=white,fill opacity=0.8,text opacity=1,inner sep=0.2pt] at (26.5,0.95) {$f(L(\cP_1){\setminus} L'(\cP))$};
\node at (30.5,0.5) {$f(L'(\cP_0))$};
\node[fill=white,fill opacity=0.8,text opacity=1,inner sep=0.2pt] at (34.5,0.5) {$f(L'(\cP_1))$};

\node[fill=white,fill opacity=0.8,text opacity=1,inner sep=0.2pt] at (0.5,3.875) {${}_{11}F(\cP) {\setminus} F'(\cP)$};
\node[fill=white,fill opacity=0.8,text opacity=1,inner sep=0.2pt] at (3.8,4.3) {${}_{10}F(\cP) {\setminus} F'(\cP)$};
\node at (7.5,3.875) {${}_{11}F'(\cP)$};
\node at (11.5,3.875) {${}_{10}F'(\cP)$};

\node[fill=white,fill opacity=0.8,text opacity=1,inner sep=0.2pt] at (23.7,3.5) {${}_1L(\cP) {\setminus} L'(\cP)$};
\node[fill=white,fill opacity=0.8,text opacity=1,inner sep=0.2pt] at (26.2,3.95) {${}_0L(\cP) {\setminus} L'(\cP)$};
\node at (30.5,3.5) {${}_1L'(\cP)$};
\node at (34.5,3.5) {${}_0 L'(\cP)$};

\node at (4,3.125) {$F(\cP_0)$};
\node at (11,3.125) {$F(\cP_1)$};
 
\node at (18,3.5) {$I(\cP)$};
\node at (18,0.5) {$f(I(\cP))$};

\node at (2,6) {\footnotesize $Q_{2n+1}$};
\node at (2,5.2) {\footnotesize $\cP$};
\node at (2,-1) {\footnotesize $f(\cP)$};

\node[red] at (3,1.90625) {\footnotesize $E^{n}$};
\node[green] at (19.2,1.90625) {\footnotesize $E^{n-1}$};
\node[blue] at (32.5,1.90625) {\footnotesize $E^{n-2}$};  
  
\end{tikzpicture}
}
\caption{Vertex sets involved in constructing the cycle factor~$\cC_{2n+1}$ in the middle four levels of~$Q_{2n+1}$.
The relevant vertex sets are drawn by rectangular boxes, where the width of the boxes represents the size of the set, drawn to scale for large values of~$n$.
The vertex sets and their images under the automorphism $f=\ol{\rev}$ are aligned vertically.
The paths $\cP=M_{2n+1,n+1}^n\cup M_{2n+1,n+1}^{n+1}$ between levels~$n+1$ and~$n+2$ and $f(\cP)$ between levels~$n-1$ and~$n$ connecting various pairs of end vertices from the vertex sets are indicated by dotted lines.
The intermediate vertices of these paths are not shown in the figure.
The additional edges between levels~$n$ and~$n+1$ from the sets $E^i\subseteq M_{2n+1,n}^i$, $i=n,n-1,n-2$, connecting those paths to a cycle factor are drawn by solid (red), dashed (green) and dash-dotted (blue) lines, respectively.
}
\label{fig:4levels}
\end{figure}

To prove that~$\cC_{2n+1}$ as defined in~\eqref{eq:def-C-eq} is indeed a cycle factor, we now consider the sets of vertices in level~$n$ that are covered by the edges from~$E$; see Figure~\ref{fig:4levels}.

\begin{lemma}
\label{lem:E-match}
For any~$n\geq 1$, the edges from~$E$ defined in~\eqref{eq:def-E} match the following sets of vertices in levels~$n+1$ and~$n$ of~$Q_{2n+1}$.
\begin{enumerate}[label=(\roman*)]
\item The sets~$F(\cP_0)$ and $f(F(\cP_0))$, $F(\cP_1)$ and $f(I(\cP))$, $I(\cP)$ and $f(F(\cP_1))$ are matched by edges from~$E^n$.
\item The sets~$I(\cP)$ and~$f(I(\cP))$ are matched by edges from~$E^{n-1}$.
\item The sets $L(\cP) \setminus L'(\cP)$ and $f(L(\cP) \setminus L'(\cP))$ are matched by edges from~$E^n$, and the sets~$L'(\cP)$ and~$f(L'(\cP))$ are matched by edges from~$E^{n-2}$.
\end{enumerate}
\end{lemma}

\begin{proof}
It suffices to show that in each of the cases, the edges from~$E$ join a vertex from one of the sets~$X$ in level~$n+1$ to a vertex from the corresponding set~$Y$ in level~$n$ (i.e., we show that these edges form an injection from~$X$ to~$Y$).
The fact that they form a surjection follows by applying the same argument to~$f(\cC_{2n+1})$, using that $f$ is an involution and that $f(M_{2n+1,n}^i)=M_{2n+1,n}^i$ for $i=n,n-1,n-2$ by Lemma~\ref{lem:lex}~(iii).

To prove the first two statements in (i), consider a vertex $x\in F(\cP)=D_{2n+1,n+1}^{=0}$ in level~$n+1$ and the corresponding lattice path (recall Lemma~\ref{lem:FLI}~(i)).
The edge of~$E_n$ flips the last $\upstep$-step starting at the abscissa, i.e., this edge joins $x=(u,1,v)$ where $u,v\in D$ are uniquely determined and $u \neq ()$ with the vertex $y=(u,0,v)=f(f(v),1,f(u))$.
If $x\in F(\cP_0)$, i.e., $v\neq ()$, then we have $y\in f(F(\cP_0))$ as $(f(v),1,f(u))\in D_{2n,n+1}^{=0}\circ 0$ (recall Lemma~\ref{lem:FLI}~(ii)).
If $x\in F(\cP_1)$, i.e., $v=()$, then we have $y\in f(I(\cP))$ as $(1,f(u))\in D_{2n+1,n+1}^{>0}$ (recall Lemma~\ref{lem:FLI}~(i)).
The remaining claims can be shown by similar calculations with the help of Lemma~\ref{lem:FLI}.
We omit the details.
\end{proof}

Lemma~\ref{lem:E-match} allows us to conclude that~$\cC_{2n+1}$ as defined in~\eqref{eq:def-C-eq} with the edge set~$E$ defined in~\eqref{eq:def-E} is indeed a cycle factor in the subgraph of~$Q_{2n+1}$ induced by the middle four levels, as every vertex in the four levels is covered by exactly two edges.

\subsection{Structure of the cycle factor~\texorpdfstring{$\cC_{2n+1}$}{C2n+1}}

We now analyze the structure of the cycle factor~$\cC_{2n+1}$ defined in~\eqref{eq:def-C-eq}.
Observe from Figure~\ref{fig:4levels} that on each cycle of~$\cC_{2n+1}$ all paths from~$\cP$ are visited in the same orientation from the first vertices~$F(\cP)$ to the corresponding last vertices from~$L(\cP)$.
The following lemma shows for a given path from~$\cP$ on a cycle~$C$ of~$\cC_{2n+1}$ which path from~$\cP$ is encountered next on the cycle~$C$.
To state the lemma we introduce a bit of notation.

It is convenient here to identify the paths from~$\cP$ by their first vertices, so by Lemma~\ref{lem:paths} this is the set $D_{2n+1,n+1}^{=0}$.
By appending an additional 0-bit to the bitstrings $D_{2n+1,n+1}^{=0}$, we obtain Dyck paths of length~$2n+2$ with exactly $n+1$~upsteps and $n+1$~downsteps that touch the abscissa at least three times (in the origin~$(0,0)$, at~$(2n+2,0)$ and at some intermediate point).
It turns out that the structure of the cycle factor~$\cC_{2n+1}$ can be described most conveniently by interpreting the Dyck paths $D_{2n+1,n+1}^{=0}\circ 0$ as rooted trees as described in Section~\ref{sec:dyck} and illustrated in Figure~\ref{fig:tree}.
We introduce the abbreviation $\cT_{n+1}:=D_{2n+1,n+1}^{=0}\circ 0$ for these trees.
Note that they have exactly $n+1$~edges and the root has degree at least two.

\begin{figure}
\includegraphics[scale=0.916]{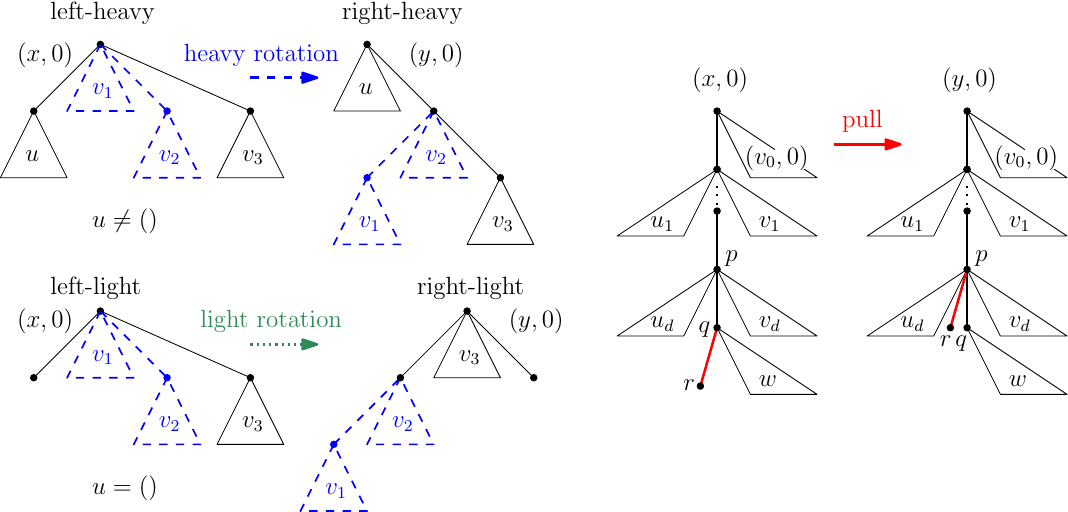}
\caption{Heavy and light tree rotations (left), and the pull operation (right).}
\label{fig:tree-ops}
\end{figure}

The following definitions are illustrated in Figure~\ref{fig:tree-ops}.
We say that a tree from~$\cT_{n+1}$ is \emph{left-light} if the leftmost child of the root is a leaf; otherwise, it is \emph{left-heavy}.
Analogously, we define \emph{right-light} or \emph{right-heavy} trees by considering the rightmost child of the root.
A \emph{left-rotation} of a tree moves the root to the leftmost child of the root.
In terms of bitstrings, this operation changes $(1,u,0,v)$ with $u,v\in D$ to $(u,1,v,0)$.
A \emph{right-rotation} is the inverse operation of a left-rotation.
Given a left-heavy tree $x=(1,u,0,s,1,v_3,0)$, $u,v_3\in D$, where $s=()$ or $s=(v_1,1,v_2,0)$ with $v_1,v_2\in D$, a \emph{heavy rotation} is a left-rotation of the tree plus a right-rotation of the subtree~$s$, i.e., the resulting tree~$y$ is given by $y=(u,1,1,v_3,0,0)$ if $s=()$ and $y=(u,1,1,v_1,0,v_2,1,v_3,0,0)$ otherwise.
Note that the resulting tree~$y$ is right-heavy.
Given a left-light tree $x=(1,0,s,1,v_3,0)$, $v_3\in D$, where $s=()$ or $s=(v_1,1,v_2,0)$ with $v_1,v_2\in D$, a \emph{light rotation} is a right-rotation of the tree plus a right-rotation of the (possibly empty) subtree~$s$ plus detaching the pending edge that leads to the leftmost child of the root of~$x$ and reattaching it as a rightmost child of the new root, i.e., the resulting tree~$y$ is given by $y=(1,0,v_3,1,0)$ if $s=()$ and $y=(1,1,v_1,0,v_2,0,v_3,1,0)$ otherwise.
Note that the resulting tree~$y$ is right-light.
To any given tree from~$\cT_{n+1}$, we can either apply a heavy or a light rotation, depending on whether the tree is left-heavy or left-light, respectively.
We refer to this mapping on~$\cT_{n+1}$ as~$\rho$.
Analogously, $\rho^{-1}$ applies either an inverse heavy or an inverse light rotation depending on whether the tree is right-heavy or right-light, respectively.

The following lemma asserts that the sequence of rooted trees corresponding to first vertices of paths from~$\cP$ that are encountered when following a cycle from our factor~$\cC_{2n+1}$ corresponds to repeatedly applying~$\rho$, i.e., either applying a heavy rotation or a light rotation.
In other words, the cycles of~$\cC_{2n+1}$ are in bijection with equivalence classes of rooted trees from~$\cT_{n+1}$ when iterating the mapping~$\rho$.

\begin{lemma}
\label{lem:tree-rot}
For any~$n\geq 1$, any cycle~$C$ of the cycle factor~$\cC_{2n+1}$ defined in~\eqref{eq:def-C-eq}, and any vertex $x\in F(\cP)=D_{2n+1,n+1}^{=0}$ let $y$ be the next vertex from~$F(\cP)$ on the cycle~$C$ encountered after~$x$.
Let $x=(1,u,0,v)$, $u\in D$, be the canonical decomposition of~$x$.
If $u\neq ()$, i.e., $x\in {}_{11}F(\cP)$, then we have
\begin{equation*}
  y=\begin{cases}
  (u,1,1,v_3,0) & \text{if } v=(1,v_3) \text{ with } v_3\in D\text{, i.e., } x \in {}_{11}F(\cP) \setminus F'(\cP) \enspace, \\
  (u,1,1,v_1,0,v_2,1,v_3,0) & \text{if } v=(v_1,1,v_2,0,1,v_3) \text{ with } v_1,v_2,v_3\in D\text{, i.e., } x \in {}_{11}F'(\cP) \enspace.
  \end{cases}
\end{equation*}
In terms of rooted trees, $(y,0)\in \cT_{n+1}$ is obtained from $(x,0)\in \cT_{n+1}$ by a heavy rotation.

If $u = ()$, i.e., $x \in {}_{10}F(\cP)$, then we have
\begin{equation*}
  y=\begin{cases}
  (1,0,v_3,1) & \text{if } v=(1,v_3) \text{ with } v_3\in D\text{, i.e., } x \in {}_{10}F(\cP) \setminus F'(\cP) \enspace, \\
  (1,1,v_1,0,v_2,0,v_3,1) & \text{if } v=(v_1,1,v_2,0,1,v_3) \text{ with } v_1,v_2,v_3\in D\text{, i.e., } x \in {}_{10}F'(\cP) \enspace.
  \end{cases}
\end{equation*}
In terms of rooted trees, $(y,0)\in \cT_{n+1}$ is obtained from $(x,0)\in \cT_{n+1}$ by a light rotation.
\end{lemma}

\begin{figure}[b!]
\makebox[0cm]{ 
\includegraphics[scale=0.916]{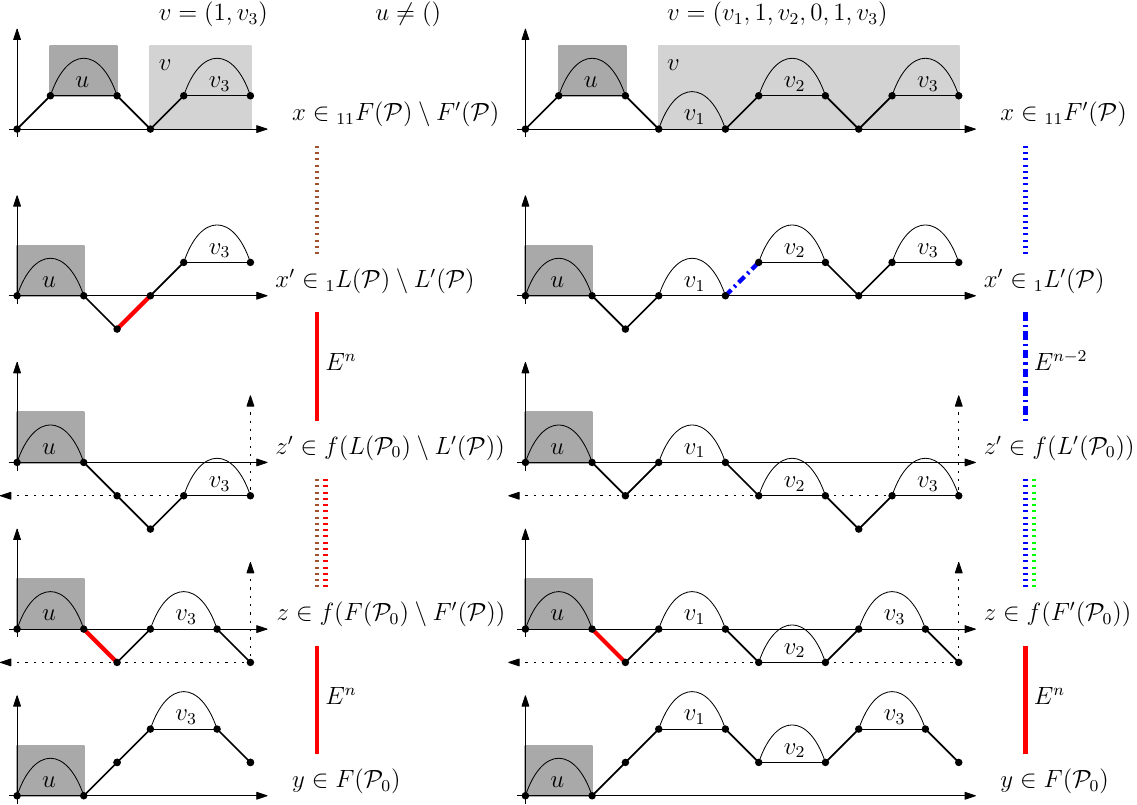}
}
\caption{Proof of the first part of Lemma~\ref{lem:tree-rot}.
The two columns represent the subpaths from~$x$ to~$y$ on a cycle from our factor for the two cases of~$v$.
Applying $f=\ol{\rev}$ can be interpreted as reading the lattice path backwards, and this is indicated by the dotted coordinate systems.
}
\label{fig:cycle-heavy}
\end{figure}

\begin{proof}
We follow the cycle~$C$ in Figure~\ref{fig:4levels} from some vertex $x\in F(\cP)$ until we encounter the next vertex $y\in F(\cP)$.
The case when $u\neq ()$ is shown in Figure~\ref{fig:cycle-heavy}.
In this case, following the path of~$\cP$ from $x\in{} _{11}{F(\cP)}$ leads to the vertex $x':=(u,0,1,v) \in{} _1L(\cP)$ by Lemma~\ref{lem:paths}.
Then, depending on whether $x'\in L(\cP)\setminus L'(\cP)$ or $x'\in L'(\cP)$, we continue along the cycle via an edge from~$E^n$ or~$E^{n-2}$, respectively, from level~$n+1$ to level~$n$.
This corresponds to the two subcases distinguished by~$v$ and shown in the left and right column in Figure~\ref{fig:cycle-heavy}.
By Lemma~\ref{lem:E-match}, we get to a vertex $z'\in f(L(\cP_0))$ that is last on some path of~$f(\cP_0)$.
By following this path backwards using Lemma~\ref{lem:paths} we get to the corresponding first vertex $z\in f(F(\cP_0))$.
Using again Lemma~\ref{lem:E-match}, we then traverse an edge from~$E^n$ to go from from level~$n$ to level~$n+1$ where we encounter the vertex $y\in F(\cP_0)$.
Observe in Figure~\ref{fig:cycle-heavy} that $y$ has exactly the claimed form.

\begin{figure}
\includegraphics[scale=0.916]{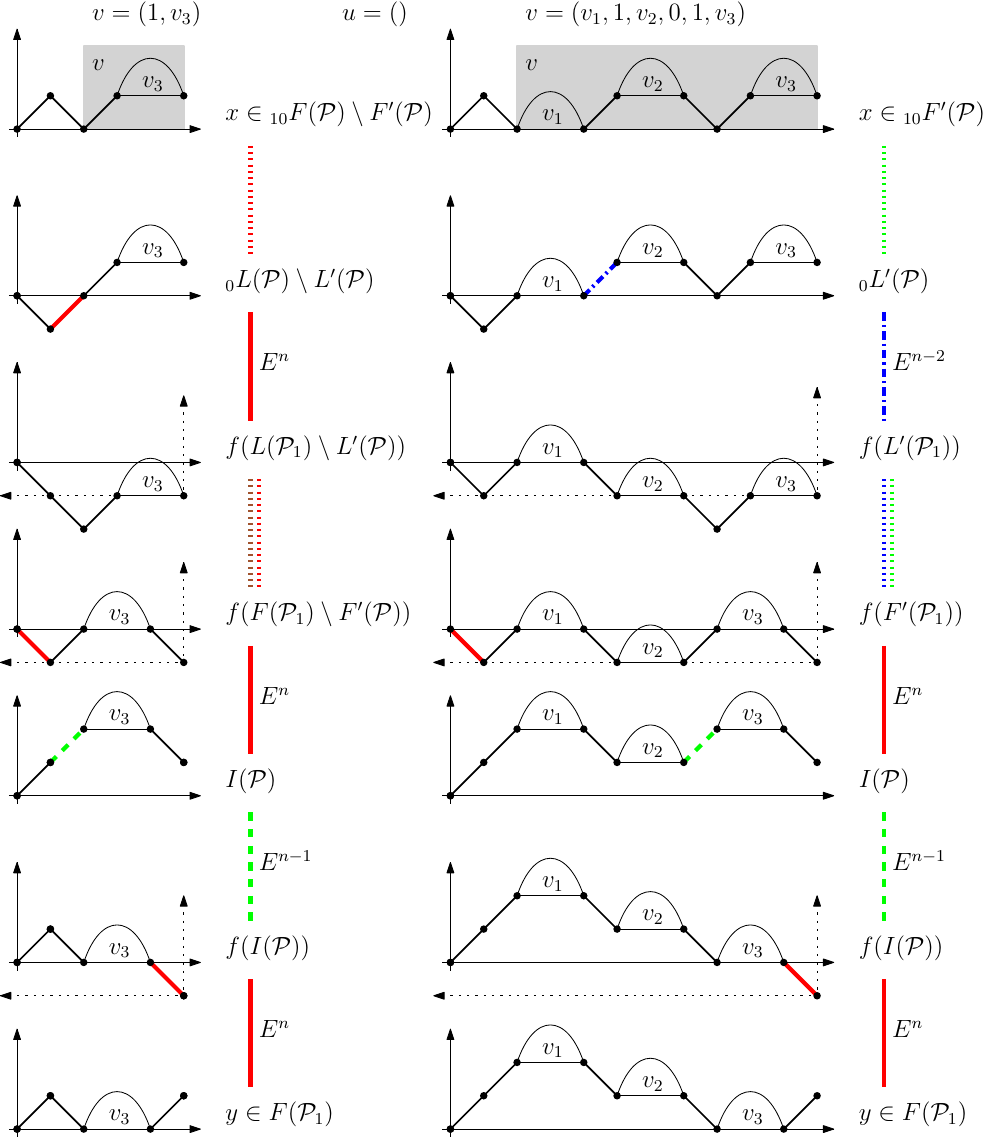}
\caption{Proof of the second part of Lemma~\ref{lem:tree-rot}.
The two columns represent the subpaths from~$x$ to~$y$ on a cycle from our factor for the two cases of~$v$.
Applying $f=\ol{\rev}$ can be interpreted as reading the lattice path backwards, and this is indicated by the dotted coordinate systems.
}
\label{fig:cycle-light}
\end{figure}

The case when $u=()$ is depicted in Figure~\ref{fig:cycle-light} and can be verified analogously.
Note that in this case we in addition visit a vertex from~$I(\cP)$ and from~$f(I(\cP))$ before coming to~$y$.
\end{proof}

The number of cycles in the factor~$\cC_{2n+1}$ is $1,1,1,4,6,19,49,150,442,1424$ for $n=1,2,\ldots,11$.
This sequence matches the first entries of~\cite{flexagons-seq}, i.e., the number of plane trivalent trees with $n$~internal vertices.
A \emph{plane trivalent tree} is a tree where every vertex has degree~1 or~3, and the neighbors of each vertex have a specified cyclic ordering.
We establish the correspondence between cycles from our factor and plane trivalent trees in the following proposition.
This special family of trees appears only in this proposition and its proof.
All subsequent arguments use again the set of rooted trees~$\cT_{n+1}$ introduced before.

\begin{proposition}
For any~$n\geq 1$, the cycles from the factor~$\cC_{2n+1}$ defined in~\eqref{eq:def-C-eq} are in bijection with the set of plane trivalent trees with $n$~internal vertices.
\end{proposition}

\begin{proof}
We first define, for any $x \in D$, binary trees~$\ell(x)$ and~$r(x)$.
If $x=()$, then~$\ell(x)$ and~$r(x)$ consist only of a single vertex.
Otherwise we write~$x$ uniquely as $x = (1,u,0,v) = (u',1,v',0)$ with $u,v,u',v'\in D$, and then $\ell(x)$ consists of a root with left child~$\ell(u)$ and right child~$r(v)$, and $r(x)$ consists of a root with left child~$\ell(u')$ and right child~$r(v')$.
Given any vertex $x\in F(\cP)=D_{2n+1,n+1}^{=0}$, we map the bitstring $x':=(x,0)$ to a trivalent tree~$\tau(x')$ rooted at one of its $n$~internal vertices as follows.
We first write $x'$ uniquely in the form $x'=(1,u,0,v,1,w,0)$ with $u,v,w\in D$, and we define $\tau(x')$ as the tree that consists of a root with left child~$\ell(u)$, middle child~$r(v)$ and right child~$r(w)$.
One can show that under this bijection~$\tau$ between $D_{2n+1,n+1}^{=0}\circ 0$ and trivalent trees rooted at one of their $n$~internal vertices, the mapping~$\rho$ corresponds to rotating the root of the trivalent tree to the leftmost child until the root is again an internal vertex.
Consequently, by Lemma~\ref{lem:tree-rot} cycles from~$\cC_{2n+1}$ correspond to equivalence classes of rooted trivalent trees whose root is one of their $n$~internal vertices under this rotation operation.
Obviously, these are exactly plane trivalent trees with $n$~internal vertices.
\end{proof}

\subsection{Flippable pairs}
\label{sec:flip}

In this section we define certain 6-cycles in the graph~$Q_{2n+1}$ between levels~$n+1$ and~$n+2$ that can be used to join pairs of cycles from our factor~$\cC_{2n+1}$ as described in the beginning of this section (see Figure~\ref{fig:c6xy}).
Let us emphasize here that all modifications of the cycle factor happen only between the two upper levels~$n+1$ and~$n+2$.

We say that two vertices $x,y \in F(\cP)=D_{2n+1,n+1}^{=0}$ form a \emph{flippable pair} $(x,y)$, if~$x$ and~$y$ have the form
\begin{equation}
\label{eq:xy}
\begin{aligned}
  x &= (1,u_1,1,u_2,\ldots,1,u_d,1,1,0,w,0,v_d,0,v_{d-1},0,\ldots,v_1,0,v_0) \enspace, \\
  y &= (1,u_1,1,u_2,\ldots,1,u_d,1,0,1,w,0,v_d,0,v_{d-1},0,\ldots,v_1,0,v_0)
\end{aligned}
\end{equation}
with~$d\geq 0$ and $u_1,\ldots,u_d,v_1,\ldots,v_d,w,(v_0,0) \in D$.
Recall that~$(x,0)$ and~$(y,0)$ can be viewed as rooted trees from~$\cT_{n+1}$.
If~$(x,y)$ is a flippable pair, then the tree~$(y,0)$ is obtained from the tree~$(x,0)$ by moving a pending edge from a vertex in the leftmost subtree to its parent.
Specifically, the pending edge~$(q,r)$ must form the leftmost subtree of a vertex~$q$ in the leftmost subtree of $(x,0)$, and this edge is removed from $q$ and reattached to the parent~$p$ of~$q$ to become the subtree directly left of the edge~$(p,q)$.
We refer to this as a \emph{pull operation}; see the right side of Figure~\ref{fig:tree-ops}.
The inverse operation takes a pending edge~$(p,r)$ in the leftmost subtree of~$(y,0)$, removes this edge from~$p$, and reattaches it as the leftmost subtree of the vertex~$q$ that is the child of~$p$ directly to the right of~$r$.

Any 6-cycle between levels~$n+1$ and~$n+2$ of~$Q_{2n+1}$ can be uniquely encoded as a string of length~$2n+1$ over $\{0,1,*\}$ with $n$ many~1s, $n-2$ many~0s and three~$*$s.
The 6-cycle corresponding to this string is obtained by substituting the three $*$s by all six combinations of at least two different symbols from~$\{0,1\}$.
We define a set of 6-cycles~$\cS_{2n+1}$ between levels~$n+1$ and~$n+2$ of~$Q_{2n+1}$ consisting of all 6-cycles
\begin{equation}
\label{eq:c6xy}
  C_6(x,y):=(u_1,0,u_2,0,\ldots,u_d,0,1,*,*,w,*,v_d,1,v_{d-1},1,\ldots,v_1,1,v_0) \enspace
\end{equation}
for a flippable pair~$(x,y)$, $x,y\in D_{2n+1,n+1}^{=0}$, as in~\eqref{eq:xy}.

Note that $C_6(x,y)\circ 00$ is a 6-cycle in the middle two levels of the $(2n+3)$-cube, and from~\cite[Proposition~3]{gregor-muetze-nummenpalo:18} we obtain the following properties.
For any $x\in D_{2n+1,n+1}^{=0}$, we write $P(x)$ for the path from the set~$\cP$ defined in~\eqref{eq:def-P} that starts at the vertex~$x$.

\begin{lemma}\label{lem:6cycles}
For any~$n\geq 1$, the 6-cycles $C_6(x,y)\in \cS_{2n+1}$ defined in~\eqref{eq:c6xy} have the following properties:
\begin{enumerate}[label=(\roman*)]
\item
Let~$(x,y)$ be a flippable pair.
The 6-cycle $C_6(x,y)$ intersects~$P(x)$ in two non-incident edges and it intersects~$P(y)$ in a single edge.
Moreover, the symmetric difference of the edge sets of the two paths $P(x)\in \cP$ and $P(y)\in \cP$ with the 6-cycle~$C_6(x,y)$ gives two paths~$P'(x)$ and~$P'(y)$ on the union of the vertex sets of~$P(x)$ and~$P(y)$, interconnecting~$x$ with the last vertex of~$P(y)$, and $y$ with the last vertex of~$P(x)$, respectively.

\item
For any flippable pairs~$(x,y)$ and~$(x',y')$, the 6-cycles~$C_6(x,y)$ and~$C_6(x',y')$ are edge-disjoint.

\item
For any flippable pairs~$(x,y)$ and~$(x,y')$, the two pairs of edges that the two 6-cycles~$C_6(x,y)$ and~$C_6(x,y')$ have in common with the path~$P(x)$ are not interleaved, but one pair appears before the other pair along the path.
\end{enumerate}
\end{lemma}

\begin{figure}
\includegraphics[scale=0.916]{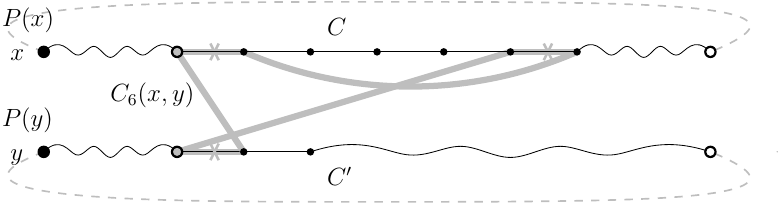}
\caption{
Two cycles from our factor joined by taking the symmetric difference with a 6-cycle.
The paths~$P(x)$ and~$P(y)$ from the set~$\cP$ (solid black) lying on the two cycles traverse the 6-cycle~$C_6(x,y)$ (solid gray) as shown.
The symmetric difference yields paths~$P'(x)$ and~$P'(y)$ that have flipped end vertices.
}
\label{fig:c6xy}
\end{figure}

\subsection{Proof of Theorem~\ref{thm:gmlc4}}

With Lemma~\ref{lem:tree-rot} and \ref{lem:6cycles} in hand, we are now ready to prove Theorem~\ref{thm:gmlc4}.

\begin{proof}[Proof of Theorem~\ref{thm:gmlc4}]
Let~$\cC_{2n+1}$ and~$\cS_{2n+1}$ be the cycle factor and the set of 6-cycles defined in Section~\ref{sec:2factor4} and \ref{sec:flip}, respectively.

Consider two cycles $C,C'$ of $\cC_{2n+1}$ containing paths $P,P'\in\cP$ with first vertices $x,y\in D_{2n+1,n+1}^{=0}$, respectively, such that $(x,y)$ is a flippable pair.
By Lemma~\ref{lem:6cycles}~(i), the symmetric difference of the edge sets~$C\cup C'$ and~$C_6(x,y)$ forms a single cycle on the vertex set of~$C\cup C'$, i.e., this joining operation reduces the number of cycles in the factor by one; see Figure~\ref{fig:c6xy}.
Recall that in terms of rooted trees, $(y,0)\in\cT_{n+1}$ is obtained from $(x,0)\in\cT_{n+1}$ by a pull operation; see Figure~\ref{fig:tree-ops}.

We repeat this joining operation until all cycles in the factor are joined to a single Hamilton cycle.
For this purpose we define an auxiliary graph~$\cH_{n+1}$ whose nodes represent the cycles in the factor~$\cC_{2n+1}$ and whose edges connect pairs of cycles that can be connected to a single cycle with such a joining operation that involves a 6-cycle from the set~$\cS_{2n+1}$.
Formally, the node set of~$\cH_{n+1}$ is given by partitioning the set~$\cT_{n+1}$ into equivalence classes under the mapping~$\rho$.
By Lemma~\ref{lem:tree-rot}, the nodes of~$\cH_{n+1}$ therefore indeed correspond to the cycles in the factor~$\cC_{2n+1}$.
Specifically, each rooted tree $(x,0)\in \cT_{n+1}=D_{2n+1,n+1}^{=0}\circ 0$ belonging to some node of~$\cH_{n+1}$ corresponds to the first vertex~$x$ of some path~$P\in \cP$ such that $P$ lies on the cycle corresponding to that node.
For every flippable pair~$(x,y)$, $x,y\in D_{2n+1,n+1}^{=0}$, we add the edge to~$\cH_{n+1}$ that connects the node containing the tree~$(x,0)$ to the node containing the tree~$(y,0)$.
By our initial argument, such a flippable pair yields a 6-cycle~$C_6(x,y)$ that can be used to join the two corresponding cycles to a single cycle.
As mentioned before, the 6-cycle lies entirely between levels~$n+1$ and~$n+2$.

To complete the proof of Theorem~\ref{thm:gmlc4}, it therefore suffices to prove that the auxiliary graph~$\cH_{n+1}$ is connected.
Indeed, if~$\cH_{n+1}$ is connected, then we can pick a spanning tree in~$\cH_{n+1}$, corresponding to a collection of 6-cycles $\cS'\seq \cS_{2n+1}$, such that the symmetric difference between the edge sets $\cC_{2n+1}\bigtriangleup \cS'$ forms a Hamilton cycle in the middle four levels of~$Q_{2n+1}$.
Of crucial importance here are properties~(ii) and~(iii) in Lemma~\ref{lem:6cycles}, which ensure that whatever subset of 6-cycles we use in this joining process, they will not interfere with each other, guaranteeing that each 6-cycle indeed reduces the number of cycles by one, as desired.

At this point we reduced the problem of proving that the middle four levels of~$Q_{2n+1}$ have a Hamilton cycle to showing that the auxiliary graph~$\cH_{n+1}$ is connected, which is much easier.
Indeed, all we need to show is that any rooted tree from~$\cT_{n+1}$ can be transformed into any other tree from~$\cT_{n+1}$ by a sequence of heavy rotations, light rotations, pulls and their inverse operations; see Figure~\ref{fig:tree-ops}.
It turns out that for our proof we only need light rotations and pulls.

Recall that heavy and light rotations correspond to following the same cycle from~$\cC_{2n+1}$ (staying at the same node in~$\cH_{n+1}$), and a pull corresponds to a joining operation (traversing an edge in~$\cH_{n+1}$ to another node).
For this we show that any rooted tree $x\in \cT_{n+1}$ can be transformed into the special tree $s:=(1,1,0,1,0,\ldots,1,0,0,1,0)\in \cT_{n+1}$, i.e., a right-light tree with a root of degree two and a star as its left subtree.
To achieve this we distinguish three cases; see Figure~\ref{fig:connect}.

\begin{enumerate}[label=(\alph*)]
\item \emph{$x$ is left-light and right-light.}
By a single light rotation we obtain a right-light tree with a root of degree two.
We then repeatedly pull pending edges in the left subtree towards the left child of the root until we end up at the tree~$s$.
    
\item \emph{$x$ is left-light and right-heavy.}
The tree obtained after applying a single light rotation is right-light.
If it is also left-light, then we continue as in case~(a).
Otherwise it is left-heavy, and then we consider the leftmost leaf of it, performing pull operations on this leaf until it becomes adjacent to the root.
We thus obtain a tree that is left-light and right-light, and then we continue as in case~(a).

\item \emph{$x$ is left-heavy.}
We consider the leftmost leaf of~$x$, and repeatedly perform pull operations on this leaf until it becomes adjacent to the root.
We obtain a left-light tree, and then we continue as in cases~(a) or~(b).
\end{enumerate}

\begin{figure}
\includegraphics[scale=0.916]{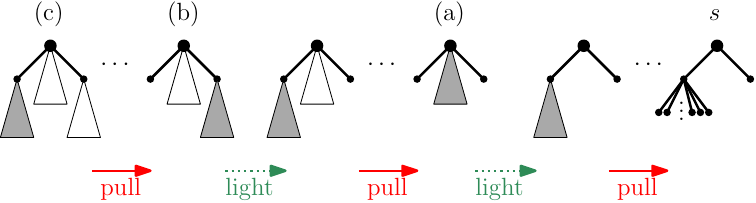}
\caption{
Transformation of trees from~$\cT_{n+1}$ into the tree~$s$ in the proof of Theorem~\ref{thm:gmlc4}.
Shaded subtrees are non-empty.
}
\label{fig:connect}
\end{figure}

This shows that~$\cH_{n+1}$ is connected, and thus completes the proof.
\end{proof}

\section{Computer experiments}
\label{sec:exp}

The numbers of cycles in the factor through the middle $2\ell$~levels of the $(2n+1)$-cube constructed as in the proof of Theorem~\ref{thm:gmlc-2fac} using the two edge-disjoint SCDs~$\cD:=\cD_0$ and~$\cD':=\ol{\cD_0}$ are shown in Table~\ref{tab:2factors1} for $n=1,2,\ldots,12$.
Note that the number of cycles in the middle two levels (the first column~$\ell=1$) seems to grow faster than the number of cycles in the middle four levels (the second column~$\ell=2$).

If instead we perform the proof of Theorem~\ref{thm:gmlc-2fac} with the two edge-disjoint SCDs constructed via our product construction, we obtain another set of cycle factors.
Specifically, denoting by~$\cD_0(a)$ the SCD~$\cD_0$ in the $a$-cube for $a\in\{2,3\}$, then $\cD:=\cD_0(3)\times \cD_0(2)^{(n-3)/2}$ and $\cD':=\ol{\cD_0(3)}\times \ol{\cD_0(2)}^{(n-3)/2}$ are two edge-disjoint SCDs in the $(2n+1)$-cube, where $\times$ denotes the product operation described in the proof of Theorem~\ref{thm:prod}, and the exponent denotes the $(n-3)/2$-fold such product.
Using those SCDs, we obtain cycle factors whose cycle lengths are shown in Table~\ref{tab:2factors2}.

\begin{table}
\caption{Number of cycles of the factor in the middle $2\ell$~levels of the $(2n+1)$-cube for $n=1,2,\ldots,12$ and $1\leq \ell\leq n+1$ arising from the proof of Theorem~\ref{thm:gmlc-2fac} using the two edge-disjoint SCDs~$\cD:=\cD_0$ and~$\cD':=\ol{\cD_0}$.}
\footnotesize
\renewcommand\tabcolsep{2pt}
\begin{tabular}{l|rrrrrrrrrrrrr}
$n$ & $\ell=1$ & 2 & 3 & 4 & 5 & 6 & 7 & 8 & 9 & 10 & 11 & 12 & 13 \\ \hline
1   & 1       & 2       &         &         &         &         &         &         &         &         &         &         &         \\
2   & 2       & 3       & 6       &         &         &         &         &         &         &         &         &         &         \\
3   & 3       & 6       & 19      & 24      &         &         &         &         &         &         &         &         &         \\
4   & 6       & 10      & 58      & 95      & 102     &         &         &         &         &         &         &         &         \\
5   & 12      & 20      & 181     & 350     & 419     & 428     &         &         &         &         &         &         &         \\
6   & 26      & 39      & 552     & 1246    & 1644    & 1749    & 1760    &         &         &         &         &         &         \\
7   & 73      & 74      & 1633    & 4292    & 6263    & 6974    & 7127    & 7140    &         &         &         &         &         \\
8   & 146     & 138     & 4750    & 14560   & 23380   & 27344   & 28546   & 28751   & 28766   &         &         &         &         \\
9   & 360     & 300     & 13500   & 48892   & 86156   & 105890  & 113477  & 115290  & 115559  & 115576  &         &         &         \\
10  & 1408    & 552     & 37716   & 163624  & 314960  & 406559  & 448446  & 461034  & 463696  & 464033  & 464052  &         &         \\
11  & 2412    & 1138    & 103998  & 547614  & 1145771 & 1551226 & 1763481 & 1838964 & 1859347 & 1863014 & 1863431 & 1863452 &         \\
12  & 10204   & 2068    & 284316  & 1836489 & 4156230 & 5892150 & 6904696 & 7315848 & 7448880 & 7479282 & 7484252 & 7484753 & 7484776 \\
\end{tabular}
\label{tab:2factors1}
\end{table}

\begin{table}
\caption{Number of cycles of the factor in the middle $2\ell$~levels of the $(2n+1)$-cube for $n=1,2,\ldots,12$ and $1\leq \ell\leq n+1$ arising from the proof of Theorem~\ref{thm:gmlc-2fac} using the two edge-disjoint SCDs $\cD:=\cD_0(3)\times \cD_0(2)^{(n-3)/2}$ and $\cD':=\ol{\cD_0(3)}\times \ol{\cD_0(2)}^{(n-3)/2}$.}
\footnotesize
\renewcommand\tabcolsep{2pt}
\begin{tabular}{l|rrrrrrrrrrrrr}
$n$ & $\ell=1$ & 2 & 3 & 4 & 5 & 6 & 7 & 8 & 9 & 10 & 11 & 12 & 13 \\ \hline
1   & 1       & 2       &         &         &         &         &         &         &         &         &         &         &         \\
2   & 2       & 3       & 4       &         &         &         &         &         &         &         &         &         &         \\
3   & 3       & 8       & 11      & 12      &         &         &         &         &         &         &         &         &         \\
4   & 10      & 22      & 34      & 39      & 40      &         &         &         &         &         &         &         &         \\
5   & 24      & 68      & 109     & 132     & 139     & 140     &         &         &         &         &         &         &         \\
6   & 80      & 213     & 362     & 456     & 494     & 503     & 504     &         &         &         &         &         &         \\
7   & 239     & 700     & 1225    & 1600    & 1779    & 1836    & 1847    & 1848    &         &         &         &         &         \\
8   & 802     & 2336    & 4222    & 5676    & 6466    & 6770    & 6850    & 6863    & 6864    &         &         &         &         \\
9   & 2638    & 7980    & 14740   & 20324   & 23662   & 25140   & 25617   & 25724   & 25739   & 25740   &         &         &         \\
10  & 9052    & 27618   & 52064   & 73330   & 87068   & 93839   & 96378   & 97084   & 97222   & 97239   & 97240   &         &         \\
11  & 31186   & 96904   & 185628  & 266344  & 321857  & 351676  & 364231  & 368320  & 369319  & 369492  & 369511  & 369512  &         \\
12  & 109460  & 343438  & 667320  & 972989  & 1194550 & 1322256 & 1381274 & 1403006 & 1409266 & 1410630 & 1410842 & 1410863 & 1410864 \\
\end{tabular}
\label{tab:2factors2}
\end{table}

\section{Open problems}
\label{sec:open}

We conclude with some interesting open problems.

\begin{itemize}
\item What are the number and length of cycles in the factors presented in Table~\ref{tab:2factors1} and Table~\ref{tab:2factors2} in terms of~$n$ and~$\ell$, and is there a combinatorial interpretation of those numbers?

\item What properties does our new SCD~$\cD_1$ have, in addition to being edge-disjoint from~$\cD_0$?
Can we exploit this construction with respect to other applications, e.g., Venn diagrams?
Are there other explicit constructions of SCDs in the $n$-cube, different from~$\cD_0$, $\cD_1$, and their complements?

\item We conjecture that the $n$-cube has $\lfloor n/2\rfloor+1$ pairwise edge-disjoint SCDs, and the best known general lower bound is~5 (recall Section~\ref{sec:outlook}).
In particular, it would be very nice to construct more than constantly many edge-disjoint SCDs in the $n$-cube as $n$~grows.
The main difficulty here is that we are missing a simple criterion like Hall's matching condition guaranteeing the existence of an SCD.
\end{itemize}

\section*{Acknowledgements}

We thank the anonymous reviewer for the careful reading and for many thoughtful comments that helped improving the presentation of this paper.

\bibliographystyle{alpha}
\bibliography{../refs}

\end{document}